\documentclass{amsart}

\usepackage{amssymb}
\usepackage{amscd}
\usepackage[mathscr]{eucal}
\usepackage{enumitem}
\usepackage{url}
\usepackage{tikz-cd}

\usepackage{hyperref}

\newtheorem{theorem}{Theorem}
\newtheorem{proposition}[theorem]{Proposition}
\newtheorem{corollary}[theorem]{Corollary}
\newtheorem{lemma}[theorem]{Lemma}
\newtheorem{fact}[theorem]{Fact}

\newtheorem{definition}[theorem]{Definition}

\newtheorem{remark}[theorem]{Remark}

\newtheorem{problem}[theorem]{Problem}
\newtheorem{question}[theorem]{Question}

\newtheorem{example}[theorem]{Example}
\newtheorem{rem/def}[theorem]{Remark/Definition}
\newtheorem{notation}[theorem]{Notation}

\numberwithin{theorem}{section}

\providecommand{\pr}{\operatorname{pr}}
\providecommand{\ch}{\mathop{char}}

\def\CC{{\mathcal C}}

\def\CH{{\mathcal H}}
\def\CL{{\mathcal L}}
\def\CM{{\mathcal M}}
\def\CN{{\mathcal N}}

\def\CR{{\mathcal R}}

\def\CU{{\mathcal U}}
\def\BN{{\mathbb N}}
\def\BZ{{\mathbb Z}}
\def\BQ{{\mathbb Q}}
\def\BR{{\mathbb R}}

\def\BF{{\mathbb F}}
\def\fm{\mathfrak{m}}

\def\CO{\mathcal{O}}

\def\dnu{\dot{\nu}}

\def\cfd{K^{\circ}}
\def\cnu{\nu^{\circ}}

\def\W{\operatorname{W}}

\def\Iso{\mathrm{Iso}}
\def\Hom{\mathrm{Hom}}

\def\id{\operatorname{Id}}
\def\ov{\overline}
\def\wi{\widetilde}
\def\L{\operatorname{L}}
\def\ob{\operatorname{Ob}}
\def\mor{\operatorname{Mor}}
\def\U{\mathrm{U}}

\def\H{\mathcal H}
\def\wH{\widehat{\H}}
\def\hp{+_{\mathcal H}}
\def\hm{\cdot_{\mathcal H}}
\def\hn{\nu_{\mathcal H}}
\def\hsum{\sum^{\mathcal H}}

\def\rg{\BZ}

\def\Tr{\operatorname{Tr}}
\def\To{\operatorname{To}}
\def\wTo{\widehat{\To}}
\def\U{\operatorname{U}}
\def\wU{\widehat{\U}}
\def\T{\operatorname{T}}
\def\res{\operatorname{Res}}
\def\wres{\widehat{\res}}
\def\W{\operatorname{W}}

\title{Hyperfields, truncated DVRs, and valued fields}
\author{Junguk Lee}
\address{Instytut Matematyczny, Uniwersytet Wrocławski\\ pl. Grunwaldzki 2/4\\        50-384 Wrocław, Poland}
\email{jlee@math.uni.wroc.pl}

\thanks{The author is supported by the Narodowe Centrum Nauki grant no. 2016/22/E/ST1/00450. The author thanks the Institut Henri Poincar\'e for their nice hospitality during the author's participating a trimester "Model Theory, Combinatorics and Valued fields" held at the IHP.  The author thanks Jaiung Jun for bringing J. Tolliver's result to the author's attention.}

\begin{document}

\begin{abstract}
For any two complete discrete valued fields $K_1$ and $K_2$ of mixed characteristic with perfect residue fields, we show that if the $n$-th valued hyperfields of $K_1$ and $K_2$ are isomorphic over $p$ for each $n\ge1$, then $K_1$ and $K_2$ are isomorphic. More generally, for $n_1,n_2\ge 1$, if $n_2$ is large enough, then any homomorphism, which is over $p$, from the $n_1$-th valued hyperfield of $K_1$ to the $n_2$-th valued hyperfield of $K_2$ can be lifted to a homomorphism from $K_1$ to $K_2$. We compute such $n_2$ effectively, which depends only on the ramification indices of $K_1$ and $K_2$. Moreover, if $K_1$ is tamely ramified, then any homomorphism over $p$ between the first valued hyperfields is induced from a unique homomorphism of valued fields. Using this lifting result, we deduce a relative completeness theorem of AKE-style in terms of valued hyperfields.
 
We also study some relationships between valued hyperfields, truncated discrete valuation rings, and  complete discrete valued fields of mixed characteristic. For a prime number $p$ and a positive integer $e$ and for large enough $n$, we show that a certain category of valued hyperfields is equivalent to the category of truncated discrete valuation rings of length $n$ and the ramification indices $e$ having perfect residue fields of characteristic $p$. Furthermore, in the tamely ramified case, we show that a subcategory of this category of valued hyperfields is equivalent to the category of complete discrete valued rings of mixed characteristic $(0,p)$ having perfect residue fields.
\end{abstract}

\maketitle

\section{Introduction}
Our main object in this article is a valued hyperfield. A hyperfield is a field-like algebraic structure whose addition is multivalued, and a valued hyperfield is a hyperfield equipped with a valuation. A typical example is the quotient of a valued field by a multiplicative subgroup of the form $1+\fm^n$ for the maximal ideal $\fm$ of the valuation ring, which is called the $n$-th valued hyperfield of a valued field. In this article, we consider a lifting problem of homomorphisms of the $n$-th valued hyperfields to homomorphisms of valued fields(See Corollary \ref{cor:general_tamely_ramifieid} and Theorem \ref{thm:mainhomlifting}). Also we study relationships between certain categories of valued hyperfields, truncated discrete valuation rings, and discrete valued fields of mixed characteristic(See Theorem \ref{thm:equiv_DVR_hvf} and Theorem \ref{thm:lifting_vhf_valuedfield}). At last, we prove a relative completeness theorem of AKE-style in terms of valued hyperfields(See Theorem \ref{thm:AKE_hyperfield}).\\

M. Krasner in \cite{Kr} introduced a notion of valued hyperfield and used it to do a theory of limits of local fields. In \cite{De}, P. Deligne did the theory of limits of local fields in a different way by defining a notion of triple, which consists of truncated discrete valuation rings and some additional data. Typical examples of a valued hyperfield and a truncated discrete valuation ring are the $n$-th valued hyperfield and the $n$-th residue ring of a valued field respectively, where the $n$-th residue ring is a quotient of a valuation ring by the $n$-th power of the maximal ideal. J. Tolliver in \cite{To} showed that discrete valued hyperfields and triples are essentially same, stated by P. Deligne in \cite{De} without a proof. In \cite{LL}, W. Lee and the author showed that given complete discrete valued fields of mixed characteristic with perfect residue fields, any homomorphism between the $n$-th residue rings of the valued fields is lifted to a homomorphism between the valued fields for large enough $n$. From this, we ask the following question.

\begin{question}\label{question:main_question_1}
Let $K_1$ and $K_2$ be discrete valued fields of mixed characteristic $(0,p)$ having perfect residue fields. Suppose the $n$-th valued hyperfields of $K_1$ and $K_2$ are isomorphic for every $n\ge 1$. Then are $K_1$ and $K_2$ are isomorphic? Moreover, is there $N>0$ such that if the $N$-th valued hyperfiels of $K_1$ and $K_2$ are isomorphic, are $K_1$ and $K_2$ isomorphic?
\end{question}

The $n$-th valued hyperfield $\H_n(K):=K/(1+\fm^n)$ of a valued field $K$ contains automatically information of the multiplicative group $K^{\times}/(1+\fm^n)$. By the multivalued addition of $\H_n(K)$, it also contains information of the $n$-residue ring of $K$ too(See Fact \ref{fact:vhf_to_triple}). We compute the multivalued addition of the $n$-th valued hyperfields rather explicitly, where the multivalued addition of given two elements in $\H_n(K)$ is given as a ball and we can compute its center and its radius of the ball explicitly(See Lemma \ref{lem:basic_on_hyperfield}). In \cite{BaKuh} and \cite{Kuh}, it was considered some structures, called {\em amc-structures}, consisted with information of the $n$-th residue rings and the multiplicative groups $K^{\times}/(1+\fm^n)$. For $n\ge 1$, the $n$-th amc structure of a valued field $K$ is a triple $K_n:=(R_K^{2n}, G_K^n, \Theta_n)$, where $R_K^n:=R/t^{2(n-1)}\fm$, $G_K^{n}:=K^{\times }/(1+t^{n-1}\fm)$, and a map $\Theta_{n}:\{x\in R_K^{2n}:\ x^2\neq 0\}\rightarrow G_K^{n},a+t^{2n}\fm\mapsto a(1+t^n\fm)$. Here, $R$ is the valuation ring of $K$, and $t=1$ if the residue field is of characteristic $0$ and $t=p$ if the residue field is of characteristic $p>0$. In \cite{BaKuh}, S. A. Basrab and F. V. Kuhlmann showed that for a valued field $K$, and for henselian valued fields $L$ and $F$ which are algebraic extension of $K$, if $L_n$ and $F_n$ are isomorphic over $K_n$ for each $n$, then $L$ and $F$ are isomorphic over $K$(See \cite[Corollary 1.4]{BaKuh}). Moreover, if $L$ and $F$ are finite extension of $K$, then it is enough to consider whether $L_n$ and $F_n$ are isomorphic over $K_n$ for large enough $n$. In \cite{Kuh}, F. V. Kuhlmann showed that if $L$ and $F$ are additionally tame extensions of $K$, then $L$ and $F$ are isomorphic over $K$ if $L_1$ and $F_1$ are isomorphic over $K_1$(See \cite[Lemma 3.1]{Kuh}).

We also study relationships between discrete valued hyperfields, truncated discrete valuation rings, and complete discrete valued fields of mixed characteristic. Fix a prime number $p$ and a positive integer $e$. Let $\CC_{p,e}$ be the category of complete discrete valuation rings of mixed characteristic $(0,p)$ having perfect residue fields and ramification index $e$. Let $\CR_{p,e}^n$ be the category of truncated discrete valuation rings of length $n$ with perfect residue fields such that $p$ is in the $e$-th power of the maximal ideal but not in the $e+1$-th power of the maximal ideal.
\begin{fact}\label{fact:motivation_question_2}\cite[Theorem 4.7, Example 3.7(2)]{LL}
For large enough $n$, there is a lifting functor $\L:\CR_{p,e}^n\rightarrow \CC_{p,e}$ satisfying several natural conditions. But two categories $\CC_{p,e}$ and $\CR_{p,e}^n$ are not equivalent.
\end{fact} 
\noindent We ask a question analogous to Fact \ref{fact:motivation_question_2} for valued hyperfields.
\begin{problem}\label{prob:main_problem_1}
Fix a prime number $p$ and a positive integer $e$. Let $n$ be a positive integer. Find a suitable category $\H_{p,e}^n$ of valued hyperfields whose has a lifting functor $\L^{\H}: \H_{p,e}^n\rightarrow \CC_{p,e}$ satisfying some natural conditions, and which makes $\H_{p,e}^n$ and $\CC_{p,e}$ equivalent. 
\end{problem}
At last, we concern a question of relative completeness of AKE-style in terms of hyper valuedfields. S. A. Basarab in \cite{Ba} showed that the theory of finitely ramified henselian valued fields of mixed characteristic is determined by the theory of the $n$-th residue ring for each $n$ and the theory of value group.
\begin{fact}\cite[Corollay 3.1]{Ba}\label{fact:Basarab_result}
Let $K_1$ and $K_2$ be finitely ramified henselian valued fields of mixed characteristic. Let $R_{1,n}$ and $R_{2,n}$ be the $n$-th residue rings of $K_1$ and $K_2$ respectively. Let $\Gamma_1$ and $\Gamma_2$ be the value groups of $K_1$ and $K_2$. Then the following are equivalent:
\begin{enumerate}
	\item $K_1\equiv K_2$.
	\item $R_{1,n}\equiv R_{2,n}$ for every $n\ge 1$ and $\Gamma_1\equiv \Gamma_2$.
\end{enumerate}
\end{fact} 
\begin{question}\label{question:main_question2}
Let $K_1$ and $K_2$ be finitely ramified henselian valued fields mixed characteristic. Let $\H_n(K_1)$ and $\H_n(K_2)$ be the $n$-th valued hyperfields rings of $K_1$ and $K_2$ respectively. If $\H_n(K_1)\equiv \H_{n}(K_2)$ for every $n\ge 1$, $K_1\equiv K_2$?
\end{question}

\bigskip

In Section \ref{preliminaries}, we recall basic notations and facts. In Section \ref{section:lifting}, we answer Question \ref{question:main_question_1} positively. We also compute such $N$ effectively. To lift a homomorphism of the $n$-th valued hyperfields, we consider homomorphisms of valued hyperfields which are over $p$. In Subsection \ref{subsection:tamely_ramifieid}, it is enough to check whether there is an isomorphism over $p$ between two first valued hyperfields in the tamely ramified case. More precisely, any homomorphism over $p$ between the first valued hyperfields is induced from a unique homomorphism between given tamely ramified complete discrete valued fields. In Subsection \ref{subsection:generally_ramified}, we compute such $N$ in Question \ref{question:main_question_1} depending only on  ramification indices of given complete discrete valued fields. To compute $N$ effectively, we first show that a homomorphism over $p$ of the $n$-th valued gives a unique homomorphism between Witt rings of given complete discrete valued fields of mixed characteristic with perfect residue fields. Using the structure theorem of totally ramified extension and Krasner's lemma, we compute $N$ effectively, which is depending only on the ramification index. In general, we can not drop the condition of being over $p$. In Section \ref{section:functoriality}, we suggest a suitable category $\wH_{p,e}^n$ of valued hyperfields whose morphisms are isometric homomorphism, and show that there is a lifting functor $\L^{\CH}:\wH_{p,e}^n\rightarrow \CC_{p,e}$ satisfying proper conditions for large enough $n$. Actually, we show that this category $\wH_{p,e}^n$ and the category $\CR_{p,e}^n$ are equivalent for large enough $n$. This essentially comes from the result of J. Tolliver in \cite{To}. We also show that the a subcategory $\H_{p,e}^n$ of $\wH_{p,e}^n$ whose morphisms are over $p$ is equivalent to $\CC_{p,e}$ for every $n>e$ if $p$ does not divide $e$. In Section \ref{section:AKE_hyperfields}, using lifting result of homomorphisms of $n$-th valued hyperfields in Section \ref{section:lifting}, we give a positive answer of Question \ref{question:main_question2} for the case of perfect residue fields. In this case, it is enough to check whether $N$-th valued hyperfields are elementary equivalent for large enough $N$. Specially, if $K_1$ and $K_2$ are tamely ramified and the first valued hyperfields are elementary equivalent, then $K_1$ and $K_2$ are elementary equivalent.

\section{Preliminaries}\label{preliminaries}
We introduce basic notations and terminologies which will be used in this paper. We denote a valued field by a tuple $(K,R(K),\fm_K,\nu_K,k(K),\Gamma_K)$ consisting of the following data : $K$ is the underlying field, $R(K)$ is the valuation ring, $\fm_K$ is the maximal ideal of $R(K)$, $\nu_K$ is the valuation map, $k_K$ is the residue field, and $\Gamma_{K}$ is the value group considering as an additive group. If there is no confusion, we omit $K$. Hereafter, the full tuple $(K,R,\fm,\nu,k,\Gamma)$ will be abbreviated in accordance with the situational need for the components. For $\gamma\in \Gamma$, we write $\Gamma_{* \gamma}:=\{x\in \Gamma|\ x* \gamma\}$ for $*\in\{\ge,>\}$ and $\fm^{\gamma}:=\{x\in K|\ \nu(x)\in \Gamma_{> \gamma}\}$. Note that for each $\gamma\in \Gamma_{\ge 0}$, $1+\fm^{\gamma}$ is a multiplicative subgroup of $K^{\times}$. For $\gamma\in \Gamma_{\ge 0}$, let $R_{\gamma}:=R/\fm^{\gamma}$. We recall the definitions of ramified valued fields.
\begin{definition}\label{def:unramified_valuedfield}
We say that a valued field $(K,\nu,\Gamma)$ is {\em (absolutely) unramified} if $\ch(k)=0$, or $\ch(k)=p$ and $\nu(p)$ is the minimal positive element in $\Gamma$ for $p>0$. We say $(K,\nu)$ is {\em (absolutely) ramified} if it is not  absolutely unramified.
\end{definition}

\begin{definition}\label{def:ramfication_index}
Let $(K,\nu,\Gamma)$ be a valued field whose residue field has prime characteristic $p$.
\begin{enumerate}
	\item We say $(K,\nu,\Gamma)$ is {\em (absolutely) finitely ramified} if the set $\{\gamma\in \Gamma|\ 0< \gamma\le \nu(p)\}$ is finite. The cardinality of $\{\gamma\in \Gamma|\ 0< \gamma\le \nu(p)\}$ is called {\em the (absolute) ramification index} of $(K,\nu)$, denoted by $e(K,\nu)$ or $e(R)$. If $K$ or $\nu$ is clear from context, we write $e(K)$ or $e$ for $e(K,\nu)$. For $x\in R$, we write $e_{\nu}(x):=|\{\gamma\in \Gamma|\ 0<\gamma\le \nu(x)\}|$. If there is no confusion, we write $e(x)$ for $e_{\nu}(x)$
	\item Let $(K,\nu,\Gamma)$ be finitely ramified. If $p$ does not divide $e_{\nu}(p)$, we say $(K,\nu)$ is \emph{(absolutely) tamely ramified}. Otherwise, we say $(K,\nu)$ is \emph{(absolutely) wildly ramified}.
\end{enumerate}
\end{definition}
\noindent Note that a discrete valued field having a residue field of characteristic $p>0$ is finitely ramified. We say that a discrete valued field $(K,\nu,\Gamma)$ with the residue field having characteristic $p>0$ is {\em normalized} if $\Gamma$ is a subgroup of $\BR$ and $\nu(p)=1$. From now on, we mean a homomorphism between valued fields is an isometric homomorphism, where a field homomorphism $f:K_1\rightarrow K_2$ is called {\em isometric} if for $a,b\in K_1$, $$\nu_{K_1}(a)<\nu_{K_1}(b)\Leftrightarrow \nu_{K_2}(f(a))<\nu_{K_2}(f(b)).$$ We recall some facts on the Witt ring and Teichm\"{u}ller representatives of complete discrete valued fields of mixed characteristic(c.f. \cite{S}).
\begin{fact}\label{fact:witt_teichmuller}
\begin{enumerate}
\item Let $W(k_1)$ and $W(k_2)$ be Witt rings of perfect residue fields $k_1$ and $k_2$ of characteristic $p>0$ respectively. Suppose that there is a homomorphism $\phi : k_1 \longrightarrow k_2 $. Then  there is  a unique lifting homomorphism $g:R_1 \longrightarrow R_2 $ such that $g$ induces $\phi$.

\item Let $R$ be a complete discrete valuation ring of characteristic $0$ with perfect residue field $k$ of characteristic $p$ and corresponding valuation $\nu$. Then
$\W(k)$ can be embedded as a subring of $R$ and $R$ is a free $\W(k)$-module of rank $\nu(p)$. Moreover,
$R=\W(k) [\pi]$ where $\pi$ is a uniformizer of $R$.

\item Let $A$ be a ring that is Hausdorff and complete for a topology defined by a decreasing sequence
$\mathfrak{a}_1 \supset \mathfrak{a}_2 \supset ...$ of ideals such that $\mathfrak{a}_n \cdot \mathfrak{a}_m \subset \mathfrak{a}_{n+m}$.
Assume that the residue ring $A_1=A/\mathfrak{a}_1 $ is a perfect field of characteristic $p$. Then:

\begin{enumerate}
\item
There exists one and only one system of representatives $h: A_1 \longrightarrow A$ which commutes with $p$-th powers:
$h(\lambda^p ) = h(\lambda)^p $. This system of representatives is called the set of Teichm\"{u}ller representatives.
\item
In order that $ a\in A $ belong to $S= h(A_1)$, it is necessary and sufficient that $a$ be a $p^n$-th power for all $n\geq 0$.
\item
This system of representatives is multiplicative which means
$$
h( \lambda \mu ) = h(\lambda)h(\mu)
$$
for all $\lambda, \mu \in A_1$.
\item
$S$ contains $0$ and $1$.
 \item
 $S\setminus\{0\}$ is a subgroup of the unit group of $A$.
\end{enumerate}
\end{enumerate}
\end{fact}
Next we introduce a notion of valued hyperfields.
\begin{definition}\label{def:vhf}\cite[Definition 1.2 and 1.4]{Kr}
\begin{enumerate}
\item A {\em hyeprfield} is an algebraic structure $(H,+,\cdot,0,1)$ such that $(H^{\times},\cdot,1)$, where $H^{\times}:=H\setminus\{0\}$, is an abelian group and there is a multivalued operation $+:\ H\times H\rightarrow 2^{H}$ for the power set $2^H$ of $H$ satisfying the followings:
\begin{enumerate}
	\item $0\cdot \alpha=0$ for all $\alpha\in H$.
	\item (Associative) $(\alpha+\beta)+\gamma=\alpha+(\beta+\gamma)$ for all $\alpha,\beta,\gamma\in H$. 
	\item (Commutative) $\alpha+\beta=\beta+\alpha$ for all $\alpha,\beta\in H$.
	\item (Distributive) $(\alpha+\beta)\cdot \gamma\subset \alpha\cdot\gamma+\beta\cdot\gamma$ for all $\alpha,\beta,\gamma\in H$. 
	\item (Identity) $\alpha+0=\{\alpha\}$ for all $\alpha\in H$.
	\item (Inverse) For any $\alpha\in H$, there is a unique $-\alpha\in H$ such that $0\in \alpha+(-\alpha)$.
	\item For all $\alpha,\beta,\gamma\in H$, $\alpha\in \beta+(-\gamma)$ if and only if $\beta\in \alpha+\gamma$.	 
\end{enumerate}
	
	\item A {\em valued hyperfield} is a hyperfield $(H,+,\cdot,0,1)$ equipped with a map $\nu$ from $H$ to $\Gamma\cup\{\infty\}$ for an ordered abelian group $\Gamma$ such that 
	\begin{enumerate}
		\item For $\alpha\in H$, $\nu(\alpha)=\infty$ if and only if $\alpha=0$;
		\item For all $\alpha,\beta\in H$, $\nu(\alpha\cdot\beta)=\nu(\alpha)+\nu(\beta)$;
		\item For all $\alpha,\beta\in H$, $\nu(\alpha+\beta)\ge \min\{\nu(\alpha),\nu(\beta)\}$;
		\item For all $\alpha,\beta\in H$, $\nu(\alpha+\beta)$ consists of single element unless $0\in \alpha+\beta$; and
		\item There is $\rho_H\in \Gamma$ such that either $\alpha+\beta$ is a closed ball of radius $\rho_H+\min\{\nu(\alpha),\nu(\beta)\}$ for all $\alpha,\beta\in H$, or $\alpha+\beta$ is a open ball of radius $\rho_H+\min\{\nu(\alpha),\nu(\beta)\}$ for all $\alpha,\beta\in H^{\times}$.
	\end{enumerate}
\end{enumerate}
For $B\subset H$ and $\alpha\in H$, define $\alpha+ B:=\bigcup_{\beta \in B}\alpha+ \beta$(*). The associativity of $+$ means that given $\alpha,\beta,\gamma\in H$, we have $(\alpha+\beta),(\beta+ \gamma)\subset H$ and $\alpha+(\beta+\gamma)=(\alpha+\beta)+\gamma$ in the sense of $(*)$. We say that $H$ is {\em discrete} if $\Gamma$ is a discrete subgroup of $\BR$.
\end{definition}
\noindent For $\alpha_0,\ldots, \alpha_k \in H$, we write $\hsum \alpha_i$ for $(\alpha_0+\cdots+ \alpha_k)\subset H$. Since the multivalued operation $+$ is associative, the notion of $\hsum$ is well-defined.
\begin{definition}\label{def:morphism_vhf}
Let $(H_i,+_i,\cdot_i,0_i,1_i,\nu_i)$ be a valued hyperfield for $i=1,2$. A map $f$ from $H_1$ to $H_2$ is called a {\em homomorphsim} if the followings hold:
\begin{enumerate}
	\item $f(0_1)=0_2$ and $f(1_2)=1_2$.
	\item $f(\alpha\cdot_1 \beta)=f(\alpha)\cdot_2 f(\beta)$ for all $\alpha,\beta \in H_1$.
	\item $f(\alpha+_1\beta)\subset f(\alpha)+_2 f(\beta)$ for all $\alpha,\beta\in H_1$.
	\item For all $\alpha,\beta \in H_1$, $$\nu_1(\alpha)\le \nu_1(\beta)\Leftrightarrow\nu_2(f(\alpha_1))\le \nu_2(f(\beta)).$$
\end{enumerate}
Let $\Hom(H_1,H_2)$ be the set of homomorphism from $H_1$ to $H_2$.
\end{definition}
\noindent Note that the definition of homomorphisms in Definition \ref{def:morphism_vhf} is weaker than one in \cite[Definition 1.5]{To}. We call homomorphisms in \cite[Definition 1.5]{To} an {\em isometric homomorphism}.

\begin{definition}\cite[Definition 1.5]{To}\label{def:isometry}
Let $(H_i,+_i,\cdot_i,0_i,1_i,\nu_i)$ be a valued hyperfield for $i=1,2$. Suppose $\nu_1(H_1^{\times})=\nu_2(H_2^{\times})$. A map $f$ from $H_1$ to $H_2$ is called an {\em isometric homomorphism} if the followings hold:
\begin{enumerate}
	\item $f(0_1)=0_2$ and $f(1_2)=1_2$.
	\item $f(\alpha\cdot_1 \beta)=f(\alpha)\cdot_2 f(\beta)$ for all $\alpha,\beta \in H_1$.
	\item $f^{-1}(\alpha+_2 \beta)= f^{-1}(\alpha)+_1f^{-1}(\beta)$ for all $\alpha,\beta \in f(H_1)$.
	\item For all $\alpha\in H_1$, $\nu_1(\alpha)=\nu_2(f(\alpha))$.
	\item $f^{-1}(1_2)$ is a ball.
\end{enumerate}
Let $\Iso(H_1,H_2)$ be the set of isometric homomorphisms from $H_1$ to $H_2$. Note that $(3)$ implies that for $\alpha,\beta\in H_1$ and $f\in \Hom(H_1,H_2)$, $f(\alpha+_1\beta)\subset f(\alpha)+_2 f(\beta)$ so that $\Iso(H_1,H_2)\subset \Hom(H_1,H_2)$.
\end{definition}

\begin{definition}\label{def:vhf_valuedhyperfield}
Let $\gamma\in\Gamma_{\ge 0}$. The {$\gamma$-hyperfield} of $K$ is a hyperfield $(K/(1+\fm^{\gamma}),\hp,\hm)$, denoted by $\H_{\gamma}(K)$, such that for each $a(1+\fm^{\gamma}),b(1+\fm^n)\in K/1+\fm^n$,
	\begin{enumerate}
		\item $a(1+\fm^{\gamma})\hm b(1+\fm^{\gamma}):=ab(1+\fm^{\gamma})$,
		\item $a(1+\fm^{\gamma})\hp b(1+\fm^{\gamma}):=\{(x+y)(1+\fm^{\gamma})|\ x\in a(1+\fm^{\gamma}),y\in b(1+\fm^{\gamma})\}$, and
	\end{enumerate}
Conventionally, we write $0$ for $0(1+\fm^{\gamma})$ and $1$ for $1(1+\fm^{\gamma})$. The valuation $\nu$ of $K$ induces a map $\hn$ on $K/(1+\fm^{\gamma})$ sending $\hn(a(1+\fm^{\gamma}))$ to $\nu(a)$. We call $(K/(1+\fm^{\gamma}),\hp,\hm,\hn)$ the {\em valued $\gamma$-hyperfield}. For $\gamma\le \lambda\in \Gamma_{>0}$, we have that $(1+\fm^{\lambda})\subset (1+\fm^{\gamma})$ and it induces a projection $\H^{\lambda}_{\gamma}:\ \H_{\lambda}(K)\rightarrow\H_{\gamma}(K),a(1+\fm^{\lambda})\mapsto a(1+\fm^{\gamma})$. We write $\H_{\gamma}:K\rightarrow \H_{\gamma}(K),x\mapsto [x]_{\gamma}$.
\end{definition}
\noindent For $A\subset K$, let $\H_{\gamma}(A):=\{a(1+\fm^{\gamma}):\ a\in A\}$ and for $a\in K$, we write $[a]_n$ for $a(1+\fm^{\gamma})\in \H_{\gamma}(K)$. If $\gamma$ is obvious, we write $[a]$ for $[a]_{\gamma}$. Note that the valuation $\hn : \H_{\gamma}^{\times}(K)\rightarrow \Gamma$ is a group epimorphism with the kernel $\H_{\gamma}(R^{\times})$.

\begin{remark}\label{rem:vhf_valuegroup}
The quotient group $\H_{\gamma}^{\times}(K)/\H_{\gamma}(R^{\times})$ is an ordered group isomorphic to $\Gamma$.
\end{remark}

\begin{definition}\label{def:over_p}
Let $K_1$ and $K_2$ be valued fields whose residue fields are of characteristic $p>0$. Let $\H_{\gamma}(K_1)$ and $\H_{\lambda}(K_2)$ be valued hyper fields of $K_1$ and $K_2$ respectively. We say a homomorphism $f:\H_{\gamma}(K_1)\rightarrow \H_{\lambda}(K_2)$ is {\em over $p$} if $f([p])=[p]$.
\end{definition}
\noindent Note that there is an isometric isomorphism which is not over $p$(See Example \ref{ex:Hom_and_HomoverZ}). In Section \ref{section:functoriality}, we see how two sets of isometric homomorphisms and of isometric homomorphisms over $p$ are different(See Theorem \ref{thm:equiv_DVR_hvf} and Theorem \ref{thm:tame_equiv_vhf_valuedfields}).

\begin{notation}\label{notation:nth_residuering_vhf}
Let $(K,\nu)$ be a finitely ramified valued field and let $\pi$ be a uniformizer of $K$. Let $n\ge m$ be positive integers. We write $\fm^n:=\fm^{\nu(\pi^{n-1})}$. We write $R_{n}(K):=R(K)/\fm^n$ and $\H_n(K):=K/(1+\fm^n)$. We call $R_n(K)$ {\em the $n$-th residue ring} of $K$ and $\H_n(K)$ {\em the $n$-th valued hyperfield} of $K$. We write $[x]_n:=x(1+\fm^n)$ for $x\in K$. We denote $\H^n_m$ and $\H_m$ for $\H^{\nu(\pi^{n-1})}_{\nu(\pi^{m-1})}$ and $\H_{\nu(\pi^{m-1})}$ respectively.  
\end{notation}

We now recall some results on hyperfields and (Deligne's) triples in \cite{To}.
\begin{definition}\cite[Definition 1.8]{To}\label{def:DVR}
A {\em truncated discrete valuation ring}, in short a truncated DVR, is a principal Artinian local ring. Let $R$ be a truncated DVR of length $l$. For $x\in R$, we define $\nu_R(x)=\sup\{ i\in \BN|\ x\in \fm_R^i \}$, where $\fm_R$ is the maximal ideal of $R$. Then $\nu_R(R)=\{ 0,1,\ldots,l-1 \}\cup \{ \infty \}$, and $\nu_R(x)=\infty$ if and only if $x=0$.
\end{definition}
\begin{definition}\cite[Definition 1.9]{To}\label{def:triples}
A {\em (Deligne's) triple} $(R,M,\epsilon)$ consists of a truncated DVR $R$, a free $R$-module $M$ of rank $1$, and a $R$-module homomorphism $\epsilon:\ M\rightarrow R$ whose image is the maximal ideal of $R$.
\end{definition}
\begin{definition}\cite[Definition 1.10]{To}\label{def:morphism_triples}
Let $T_1=(R_1,M_1,\epsilon_1)$ and $T_2=(R_2,M_2,\epsilon_2)$ be triples. A morphism of triples $(r,f,\eta):T_1\rightarrow T_2$ consists of a homomorphism $f:R_1\rightarrow R_2$, an integer $r$, called the ramification index, and an $R_1$-module homomorphism $\eta :M_1\rightarrow M_2^{\otimes r}$ such that \begin{itemize}
	\item $f\circ \epsilon_1=\epsilon_2\circ \eta$; and
	\item $\eta$ induces a $R_2$-module isomorphism from $M_1\otimes_{R_1} R_2$ to $M_2^{\otimes r}$.
\end{itemize}
The composition of morphisms $(r_1,f_2,\eta_1)$ and $(r_2,f_2,\eta_2)$ of triples is given by $$(r_1,f_2,\eta_1)\circ (r_2,f_2,\eta_2)=(r_1r_2,f_1\circ f_2, \eta_1^{\otimes r_2}\circ \eta_2).$$
\end{definition}
\begin{remark}\cite[Remark 4.5]{To}
Let $T=(R,M,\epsilon)$ be a triple. Let $M$ be a free $R$-module of rank $1$ and $\Pi$ be a generator. For each $k\in \BZ$, the tensor power $M^{\otimes k}$ is a well-defined $R$-module of rank $1$. More precisely, $$M^{\otimes k}=\begin{cases}
R(\Pi^{\otimes k})& \mbox{if }k>0\\
R& \mbox{if }k=0\\
\Hom_R(M^{\otimes (-k)}, R)= R(\Pi^{\otimes k})& \mbox{if }k<0
\end{cases},$$
where for $k<0$, $\Pi^{\otimes k}\in \Hom_R(M^{\otimes (-k)}, R)$ is a unique homomorphism sending $\Pi^{\otimes (-k)}$ to $1$. Define a map $\nu_T:\ \bigcup_{k\in \BZ}\limits M^{\otimes k}\rightarrow \BZ\cup \{\infty\}$, called a {\em valuation map} of $T$, as follows: For $x=r\Pi^{\otimes k}\in M^{\otimes k}$, $\nu_T(x)=\nu_{R}(r)+k$.
\end{remark}

\begin{notation}
Let $(H,\nu)$ be a discrete valued hyperfield. Denote $\theta_H:=\min\{\nu(x)|\ \nu(x)>0,\ x\in H\}$.
\end{notation}
\begin{rem/def}\cite[Remark 4.1]{To}\label{rem/def:length_vhf}
Let $H$ be a discrete valued hyperfield. There is a positive integer $l$ such that $\rho_H=l\theta_H$. Such $l$ is called the {\em length} of $H$, denoted by $l(H)$.
\end{rem/def}
\begin{example}
For a discrete valued field $K$, we have that $l(\H_n(K))=n$.
\end{example}

\begin{definition}\label{def:distance_equiv}
Let $(H,\nu)$ be a discrete valued hyperfield. For $\eta\in \BR$, define an equivalence relation $\equiv_{\eta}$ on $H$ as follows: For $\alpha,\beta\in H$, $\alpha\equiv_{\eta} \beta$ if and only if $\nu(\alpha-\beta)\ge \eta$. For $\alpha\in H$, we write $[\alpha]_{\eta}$ for the $\equiv_{\eta}$-class of $\alpha$. If $\eta$ is obvious, we omit it. Denote $\CO_H:=\{\alpha\in H|\ \nu(\alpha)>0\}$ and $\fm_H:=\{\alpha\in H|\ \nu(\alpha)\ge \theta_H\}$.
\end{definition}

\begin{fact}\cite[Sections 4 and 5]{To}\label{fact:vhf_to_triple}
Let $H$ be a discrete valued hyperfield. Let $R=\CO_H/\equiv_{\rho_H}$, $M=\fm_H/\equiv_{\rho_H+\theta_H}$, and $\epsilon:\ M\rightarrow R,\ [\alpha]_{\rho_H+\theta_H}\mapsto [\alpha]_{\rho_H}$ for $\alpha\in H$. Then,
\begin{enumerate}
	\item $\Tr(H)=(R,M,\epsilon)$ is a triple, and
	\item $R$ is of length $l(H)$ and the maximal ideal $\fm_R$ is $M/\equiv_{\rho_H}$.
\end{enumerate}
We write $R_l(H)$ for $R$. If $H$ or $l$ is obvious, we write $R_l$ or $R$ for $R_l(H)$. For discrete valued hyperfields $H'$ and $H"$, and for $f\in \Iso(H,H')$ and $g\in \Iso(H',H'')$, we have that
\begin{enumerate}
	\item $\Tr(f)$ and $\Tr(g)$ are homomorphisms between $\Tr(H)$, $\Tr(H')$, and $\Tr(H'')$ respectively, and
	\item $\Tr(g\circ f)=\Tr(g)\circ \Tr(f)$. 
\end{enumerate}
\end{fact}
\begin{rem/def}\label{rem/def:residuefield_vhf}
Let $H$ be a discrete valued hyperfield and $\Tr(H)=(R,M,\epsilon)$. The set $\CO_H/\equiv_{\theta_H}$ forms a field and it is isomorphic to $R/\fm_R$. The field $\CO_H/\equiv_{\theta_H}$ is called the {\em residue field} of $H$, denoted by $k(H)$. Moreover, for each $g\in \Iso(H_1,H_2)$, it induces a homomorphism $k(g):k(H_1)\rightarrow k(H_2)$.
\end{rem/def}
\begin{proof}
To show $\CO_H/\equiv_{\theta_H}$ forms a field, mimic the proof of \cite[Lemma 4.2]{To}. And consider a map $f:\CO_H/\equiv_{\theta_H}\rightarrow R$ sending $[\alpha]_{\theta_H}$ to $[\alpha]_{\rho_H}+\fm_R$. Since $\theta_H\le \rho_H$, $f$ is well-defined and it induces an isomorphism. Let $H_1$ and $H_2$ be discrete valued hyperfield and let $g:H_1\rightarrow H_2$ be an isometric homomorphism. Then it induces $\Tr(g):\Tr(H_1)\rightarrow \Tr(H_2)$, which induces a homomorphism $k(g)$ from $k(H_1)$ to $k(H_2)$.
\end{proof}

\begin{remark}\cite[Section 6]{To}\label{rem:triple_to_vhf}
Let $T=(R,M,\epsilon)$ be a triple with the valuation map $\nu_T$. Define $\U(T):=\{0\}\cup\bigcup_{i\in \BZ}\limits \{x\in M^{\otimes i}|\ \nu_T(x)=i\}$. Then $\U(T)$ is a discrete valued hyperfield. Moreover, if $T=\Tr(H)$ for a discrete valued hyperfield $H$, then $\U(T)\cong H$ by an isometric isomorphism after rescaling $\nu_T(\Pi)=\theta_H$ for a generator $\Pi$ of $M$. Moreover the assignments $\Tr$ and $\U$ are funtorial. For discrete valued hyperfields $H_1$ and $H_2$, each $f\in\Iso(H_1,H_2)$ induces a morphism $\Tr(f)\in\Hom(\Tr(H_1),\Tr(H_2))$. For triples $T_1$ and $T_2$, each $g\in \Hom(T_1,T_2)$ induces a morphism $\U(g)\in \Iso(\U(T_1),\U(T_2))$.  
\end{remark}

\begin{example}\cite[Example 4.9]{To}\label{ex:vhf_of_valuedfield}
Let $K$ be a discrete valued field. Let $\H_n(K)$ be the $n$-th valued hyperfield of $K$ for a positive integer $n$. Then $\Tr(\H_n(K))\cong (R(K)/\fm_K^n,\fm_K/\fm_K^{n+1},\epsilon)$, where $\fm_K$ is the maximal ideal of the valuation ring $R(K)$ and the map $\epsilon$ is induced by the inclusion $\fm_K\subset R(K)$.
\end{example}

We introduce some algebraic and model theoretic structural theorems in terms of the $n$-th valued rings for finitely ramified valued fields.
\begin{rem/def}\label{rem/def:krasnernumber}
Let $R$ be a complete discrete valuation ring of mixed characteristic. Let $\pi$ be a uniformizer of $R$ and $\nu$ corresponding valuation of $R$. Let $L$ and $K$ be the fraction fields of $R$ and $\W(k)$ respectively.
We denote the maximal number
$$
 \max\left\{\nu\big(\pi- \sigma(\pi) \big): \sigma\in \Hom_{K}(L,L^{alg}), \sigma (\pi) \neq \pi \right\}
$$
by $M(R)_{\pi}$ or $M(L)_{\pi}$. This $M(R)_{\pi}$ does not depend on the choice of $\pi$ and we write $M(R)=M(L)$.

Let $(F,\nu_F)$ be a complete discrete valuation ring of mixed characteristic having the same ramification index with $L$. Suppose there is an homomorphism from $L$ to $F$. Then $M(F)=M(L)$
\end{rem/def}

\begin{remark}\cite[Theorem 3.9]{LL}
Let $(K_1,\nu_1)$ and $(K_2,\nu_2)$ be complete discrete valued fields of mixed characteristic $(0,p)$. Then we have that $$M(K_1)e_{\nu_1}(p)e_{\nu_2}(p)\le e_{\nu_2}(p)(1+e_{\nu_1}^2(p)).$$
\end{remark}

\noindent To determine whether two complete discrete valued fields of mixed characteristic with perfect residue fields are isomorphic, it is enough to check whether the $n$-th residue rings are isomorphic for large enough $n$.
\begin{fact}\cite{LL}\label{fact:lift_n-th_residue}
Let $K_1$ and $K_2$ be complete discrete valued fields of mixed characteristic with perfect residue fields. Let $R_{1,n}$ and $R_{2,n}$ be the $n$-th residue rings of $K_1$ and $K_2$ respectively. For $n> e_{\nu_2}(p)(1+e_{\nu_1}^2(p))$, if $R_{1,n}$ and $R_{2,n}$ are isomorphic, then $K_1$ and $K_2$ are isomorphic. More generally, for positive integers $n_1$ and $n_2$ with $n_2>e_{\nu_2}(p)(1+e_{\nu_1}^2(p))$, any homomorphism from $R_{1,n_1}$ to $R_{2,n_2}$ is lifted to a homomorphism from $K_1$ to $K_2$.
\end{fact}
For model theory of valued fields, we take the following languages for valued fields and their related structures. Let $\CL_{K}=\{+,-,\cdot;0,1;|\}$ be a ring language with a binary relation $|$ for valued fields, where we interpret the binary relation $|$ as $a|b$ if $\nu(a)\le \nu(b)$ for $a,b\in K$. Let $\CL_{k}=\{+',-',\cdot';0',1'\}$ be the ring language for residue fields, and $\CL_{\Gamma}=\{+^*;0^*;<\}$ be the ordered group language for value groups. For each $n\le 1$, let $\CL_{R_n}=\{+_n,-_n,\cdot_n;0_n,1_n\}$ be the ring language for the $n$-th residue ring. For $n=1$, we identify $\CL_{R_1}=\CL_{k}$. We use a language $\CL_{vhf}:=\{0,1,\cdot,+,|\}$ for valued hyperfields, where $0,1$ are constant symbols, $\cdot$ is a binary function symbol, $+$ is a ternary predicate, and $|$ is a binary relation. For a valued hyperfield $(H,\cdot,+,\nu)$, $1^H$ is interpreted as the identity of the multiplication, $0^H$ as the identity of the addition, $\cdot^H$ as the multiplication function on $H$. For $\alpha,\beta,\gamma\in H$, $(\alpha,\beta,\gamma)\in +^H$ if and only if $\gamma\in \alpha+\beta$, and $(\alpha,\beta)\in |^H$ if and only if $\nu(\alpha)\le \nu(\beta)$. For the convention, $(\alpha,0^H)\in |^H$ for all $\alpha$ in $H$. 

J. Ax and S. Kochen in \cite{AK}, and Y. Ershov in \cite{E} independently proved that the first order theories of unramifieid valued fields of characteristic $0$ are determined by the first order theories of their residue fields and valued groups.
\begin{fact}\label{fact:AKE}\cite{AK,E}{(Ax-Kochen-Ershov principle)} Let $(K_1,k_1,\Gamma_1)$ and $(K_2,k_2,\Gamma_2)$ be unramified henselian valued fields of characteristic zero. Then we have that
\begin{center}
$K_1\equiv K_2$ if and only if $k_1\equiv k_2$ and $\Gamma_1\equiv \Gamma_2$.
\end{center}
\end{fact}
\noindent In \cite{Ba}, S. A. Basarab generalized Fact \ref{fact:AKE} to the case of finitely ramified valued fields and in \cite{LL}, W. Lee and the author improved the result of Basarab in the case of perfect residue fields.
\begin{fact}\cite{Ba,LL}\label{fact:fined_AKE_Rn}
Let $(K_1,\nu_1,k_1,\Gamma_1)$ and $(K_2,\nu_2,k_2,\Gamma_1)$ be henselian valued fields of mixed characteristic $(0,p)$ with finite ramification indices. Suppose $k_1$ and $k_2$ are perfect fields. Let $n> e_{\nu_2}(p)(1+e_{\nu_1}^2(p))$. The following are equivalent:
		\begin{enumerate}
			\item $K_1\equiv K_2$;
			\item $\Gamma_1\equiv \Gamma_2$ and $R_n(K_1)\equiv R_n(K_2)$.
		\end{enumerate}
\end{fact}

\section{Lifting}\label{section:lifting}
In this section, we aim to show that for large enough $n$, any homomorphism over $p$ between the $n$-th valued hyperfields of discrete complete valued fields can be lifted to a homomorphism of given valued fields.

\begin{lemma}\label{lem:basic_on_hyperfield}
Let $a,b\in K$ and $a_0,\ldots,a_k\in K$. Fix $\gamma\in \Gamma_{\ge 0}$.
\begin{enumerate}
	\item $\bigcup [a]_{\gamma}=a(1+\fm^{\gamma})=\{x|\ \nu(x-a)\ge \gamma+\nu(a)\}$.
	\item $\bigcup [a]_{\gamma}\hp [b]_{\gamma}=\{x\in K|\ \nu(x-(a+b))> \gamma+\min\{\nu(a),\nu(b)\}\}$.
	\item $0\in \bigcup [a]_{\gamma}\hp [b]_{\gamma}$ if and only if $\bigcup [a]_{\gamma}\hp [b]_{\gamma}=\fm^{\gamma+\min\{\nu(a),\nu(b)\}}$. 
		\item $(a_0+\ldots+a_k)\in \hsum [a_i]_{\gamma}$.
	\item Suppose $b\in \bigcup \hsum [a_i]_{\gamma}$ and $a_0,\ldots,a_k\in R$. Then $b=(a_0+\ldots+a_k)+d$ for some $d\in \fm^{\gamma}$.

\end{enumerate}
\end{lemma}
\begin{proof}
Let $a,b\in K$ and $a_0,\ldots,a_k\in K$.

(1) For $x\in K$, we have that \begin{align*}
x\in a(1+\fm^{\gamma})&\Leftrightarrow x=a+ad,\ d\in \fm^{\gamma}\\
&\Leftrightarrow \nu(x-a)=\nu(ad)=\nu(a)+\nu(d), d\in \fm^{\gamma}\\
&\Leftrightarrow \nu(x-a)> \nu(a)+\gamma.
\end{align*}\\

(2) ($\subseteq$) Let $x\in \bigcup [a]_{\gamma}\hp [b]_{\gamma}$ so that thee are $c,d\in \fm^{\gamma}$ such that $x=a+b+ac+bd$. Then we have that $$\nu(x-(a+b))=\nu(ac+bd)\ge \min\{\nu(ac),\nu(bd)\}> \gamma+\min\{\nu(a),\nu(b)\}.$$

($\supseteq$) Take $x\in K$ such that $\nu(x-(a+b))>\gamma+\min\{\nu(a),\nu(b)\}$. WLOG we may assume that $\nu(a)\ge \nu(b)$. Then there is $c\in \fm^{\gamma}$ such that $x=(a+b)+bc=a+b(1+c)$, and $[x]_{\gamma}\in[a]_{\gamma}\hp[b]$. So, we conclude that $x\in \bigcup [a]_{\gamma}+_G [b]_{\gamma}$.\\

(3) By (2), $\bigcup [a]_{\gamma}\hp [b]_{\gamma}=(a+b)+\fm^{{\gamma}+\min\{\nu(a),\nu(b)\}}$. Thus we have that
\begin{align*}
0\in \bigcup [a]_{\gamma}\hp [b]_{\gamma}&\Leftrightarrow \nu(a+b)> {\gamma}+\min\{\nu(a),\nu(b)\}\\
&\Leftrightarrow (a+b)\in \fm^{{\gamma}+\min\{\nu(a),\nu(b)\}}\\
&\Leftrightarrow \bigcup [a]_{\gamma}\hp [b]_{\gamma}=\fm^{{\gamma}+\min\{\nu(a),\nu(b)\}}.
\end{align*}

(4) and (5) come from (2).
\end{proof}

\begin{proposition}\label{prop:over_p=over_Z}
Let $K_1$ and $K_2$ be valued fields whose residue fields are of characteristic $p>0$. Let $\H_{\gamma}(K_1)$ and $\H_{\lambda}(K_2)$ be valued hyper fields of $K_1$ and $K_2$ respectively. Let $f\in \Hom(\H_{\gamma}(K_1),\H_{\lambda}(K_2))$ be over $p$. Then for all $n\in \BZ$, $f([n])=[n]$. We denote $\Hom_{\BZ}(\H_{\gamma}(K_1),\H_{\lambda}(K_2))$ for the set of all homomorphisms over $p$.
\end{proposition}
\begin{proof}
Let $f\in \Hom(\H_{\gamma}(K_1),\H_{\lambda}(K_2))$ be over $p$. First, we have $f([1])=[1]$ because $[1]$ is the multiplicative identity. Since $[-1]$ is the additive inverse of $[1]$, $f([-1])=[-1]$. By Lemma \ref{lem:basic_on_hyperfield}(2), we have $f([k])=[k]$ for $1\le k\le p-1$ inductively. Now choose an integer $n$ arbitrary. Since $f([-1])=[-1]$ and $f([0])=[0]$, we may assume that $n> 0$. Suppose $n$ and $p$ are relatively prime. We write $n=a_0+a_1p+a_2+\ldots+a_mp^m$ with $0\le a_i\le p-1$ and $a_0\neq 0$ for some $m\ge 0$. Suppose $f([n])=[n']$ for some $n'\in K_2$. Then we have that
\begin{align*}
[n']&=f([n])=f([\sum_{i}\limits a_ip^i])\\
&\in \sum_i^{\H}f([a_ip^i])= \sum_i^{\H}f([a_i])f([p])^i=\sum_i^{\H}[a_i][p]^i=\sum_i^{\H}[a_ip^i].
\end{align*}
By Lemma \ref{lem:basic_on_hyperfield}(5), we have that $n'=\sum_i a_ip^i + d=n+d$ for some $d\in \fm_2^{\lambda}$. Since $n\notin \fm_2$, we have that $n'/n=1+d/n\in 1+\fm_2^{\lambda}$ and $n'(1+\fm_2^{\lambda})=n(1+\fm_2^{\lambda})$. Now suppose $p$ divides $n$. Write $n=n_0p^l$ for some $l>0$ and for some $n_0$ coprime to $p$. Then we have that $f([n])=f([n_0p^l])=f([n_0])f([p^l])=[n_0][p]^l=[n_0p^l]=[n]$.
\end{proof}

\begin{remark}\label{rem:vf_rigid_hom}
Let $(K_1,\nu_1,k_1)$ and $(K_2,\nu_2, k_2)$ be finitely ramified valued fields  having the same ramification index $e$. Let $p=\ch(k_1)=\ch(k_2)>0$. Suppose $\nu_1$ and $\nu_2$ are normalized, that is, $\nu_1(p)=\nu_2(p)=1$. Then for any $n,m\ge 1$ and $f\in \Hom_{\rg}(\H_n(K_1),\H_m(K_2))$, we have that $\nu_1(\alpha)=\nu_2(f(\alpha))$ for every $\alpha\in \H_n(K_1)$. And any homomorphism from $K_1$ to $K_2$ induces an isometric homomorphism from $\H_n(K_1)$ to $\H_n(K_2)$.
\end{remark} 

In the remaining part of this section, {\bf we assume that a complete discrete complete valued field of mixed characteristic $(0,p)$ has the normalized valuation so that $\nu(K^{\times})\subset \BR$ and $\nu(p)=1$}. For a henelian valued field $(K,\nu)$ of characteristic $0$, there is a unique valuation on $K^{alg}$ extending $\nu$ and we use the same notion $\nu$ for this valuation on $K^{alg}$.

\begin{remark}\label{rem:H1(S)_resideufield}
Let $S$ be the set of Teichm\"{u}ller representatives of a complete discrete valued field $K$ of mixed characteristic having a perfect residue field. Then, $\H_1(S)$ is a field which is isomorphic to the residue field of $K$.
\end{remark}
\begin{proof}
Let $H=(\H_1(K),+_1,\times_1)$ be the first valued hyperfield of $K$. For each $a,b\in S$, there is a unique $c\in S$ such that $[c]\in [a]+_1[b]$. So, $(\H_1(S),+,\times_1)$ forms a field where for $a,b\in S$, $[a]+[b]=[c]$ if $[c]\in [a]+_1[b]$. Consider a map sending $[a]$ to $a+\fm$, where $\fm$ is the maximal ideal of the valuation ring of $K$ and this induces an isomorphism from $\H_1(S)$ to the residue field of $K$.
\end{proof}

\begin{lemma}\label{lem:preserving_Teichmuller}
Let $K_1$ and $K_2$ be complete discrete valued fields of mixed characteristic having perfect residue fields. Let $S_1$ and $S_2$ be the set of Teichm\"{u}ller representatives of $K_1$ and $K_2$ respectively, and let $f\in \Hom(\H_{n}(K_1),\H_{m}(K_2))$. 
\begin{enumerate}
	\item We have $f(\H_{n}(S_1))\subset \H_{m}(S_2)$ and $f\restriction_{ \H_n(S_1)}$ is injective.
	\item If $f$ is over $p$, then $f\restriction_{\H_n(\W(k_1))}$ is induced by a unique homomorphism from $\W(k_1)$ to $\W(k_2)$, where $\W(k_1)$ and $\W(k_2)$ are Witt subrings of $K_1$ and $K_2$ respectively.
\end{enumerate}
\end{lemma}
\begin{proof}
(1) The proof is similar to the proof of \cite[Lemma 3.3]{LL}. Let $R_1$ and $R_2$ be the valuation rings of $K_1$ and $K_2$ respectively. By Fact \ref{fact:witt_teichmuller}(3), $S_i$ is contained in $\W(k_i)^{\times}$ where $k_i$ is the residue field of $K_i$ for $i=1,2$. For each $\lambda \in S_1$, let $\eta_{s}$ be any representative of $f(\lambda^{1/p^s}(1+\fm_1^n))$ so that $(\eta_s)^{p^s}(1+\fm_2^m)=f(\lambda(1+\fm_1^n))$. Since $\eta_s$ is in $\W(k_2)^{\times}$, we have that $\eta_s(1+\fm_2^m)=\eta_s+\fm_2^m$ and $(\eta_s)^{p^s}(1+\fm_2^m)=(\eta_s)^{p^s}+\fm_2^m$ as a set. For any other representative $\theta_s$ of $f(\lambda^{1/p^s}(1+\fm_1^n))$, we have that $\eta_s+\fm_2^m=\theta_s+\fm_2^m$. If we write $\eta_s= \theta_s + \pi_2 ^{n_2}a$ for some $a$ in $R_2$, the following binomial expansion
\begin{align*}
\eta_s ^{p^s} &=(\theta_s + \pi_2 ^{n_2}a)^{p^s}\\
{}&=\theta_s ^{p^s} + p^s \theta_s ^{ p^s -1} \pi_2^{n}a+...+(\pi_2 ^{n}a)^{p^s}
\end{align*}
shows
$\eta_s ^{p^s} -\theta_s ^{p^s} \in \fm_2 ^{s} $. Since $\eta_{s+1} ^p$ is a representative of $ f(\lambda ^{1/{p^s}}(1 +\fm_1^{n}))$, the calculation above shows that $(\eta_s ^{p^s})$ is a Cauchy sequence and
$\lim_{s\rightarrow \infty} \eta_s ^{p^s}$ is well-defined in $R_2$. Since $ \eta_s ^{p^s}(1 +\fm_2^{m}) =f(\lambda(1+\fm_1^{n}))$ and $1+\fm_2^{m}$ is topologically closed in $R_2$,
$$
f\left(\lambda(1+\fm_1^{n}) \right)= \left(\lim_{s\rightarrow \infty} \eta_s ^{p^s}\right)(1 +\fm_2^{m}).
$$
Similarly, we have
$$
f\left((\lambda)^{1/p}(1+\fm_1^{n}) \right)= \left(\lim_{s\rightarrow \infty} \eta_s ^{p^{s-1}}\right)(1 +\fm_2^{m}).
$$
Since
$$
\lim_{s\rightarrow \infty} \eta_s ^{p^s} =
\left(\lim_{s\rightarrow \infty} \eta_s ^{p^{s-1}}\right)^p,
$$
we obtain
$$
\lim_{s\rightarrow \infty} \eta_s ^{p^s}\in S_2
$$
by \ref{fact:witt_teichmuller}(3). Therefore, we have that $f(\H_n(S_1))\subset \H_m(S_2)$.

By Remark \ref{rem:H1(S)_resideufield}, the restriction map $f\restriction \H_1(S_1)$ induces an homomorphism between residue fields of $K_1$ and $K_2$. So, we have that $f\restriction \H_1(S_1)$ is trivial or injective. Since $1\in S_1$, we have that $f([1]_n)=[1]_m$ and $f\restriction \H_1(S_1)$ is not trivial. Therefore, $f\restriction \H_1(S_1)$ must be injective.\\

(2) Suppose $f$ is over $p$, that is, $f([p]_n)=[p]_m$. Note that each $a\in \W(k_i)$ is uniquely written as $\sum_{k\ge 0}\limits a_kp^k$ for  $a_k\in S$ for $i=1,2$. By Fact \ref{fact:witt_teichmuller}(1), Remark \ref{rem:H1(S)_resideufield}, and $(1)$, we have a homomorphism $\bar f:\W(k_1)\rightarrow \W(k_2)$ such that $f([a]_n)=[\bar f(a)]_m$ for $a\in S_1$. Take $a \in \W(k_1)$ and write $a=p^l\sum_{k\ge 0}\limits a_k p^k$ for $l\ge 0$ and $a_k\in S_1$ with $a_0\neq 0$. We have that
\begin{align*}
f([a]_n)&=[p]_m^lf([\sum_{k\ge 0} a_k p^k])_m\\
&=[p]_m^lf([a_0+a_1p+a_2p^2+\cdots+a_sp^s])_m,
\end{align*}
where $s=\max\{n,m\}$. It is enough to show that $$f([\sum_{0\le k\le s}\limits a_kp^k]_n)=[\sum_{0\le k\le s}\limits \bar f(a_k)p^k]_m$$ for $a_0\neq 0$ and $s\ge \max\{n,m\}$. Take $a=\sum_{0\le k\le s}\limits a_kp^k\in \W(k_1)$ with $a_0\neq 0$ and $s\ge \max\{n,m\}$. Then,
\begin{align*}
f([a])&\in \sum_{0\le k\le s}^{\H}\limits f([a_k]_n)[p]_m^k\\
&=\sum_{0\le k\le s}^{\H}\limits [\bar f(a_k)]_m[p]_m^k\\
&=\sum_{0\le k\le s}^{\H}\limits [\bar f(a_k)p^k]_m.
\end{align*}
So, we have that 
\begin{align*}
f([a]_n)&=[\sum_{0\le k\le s}\limits \bar f(a_k)p^k+d]_m\\
&=[\bar f(a)+d]
\end{align*}
for some $d\in \fm_2^m$ by Lemma \ref{lem:basic_on_hyperfield}(5). Since $(\bar f(a)+d)/\bar f(a)=1+d/\bar f(a)\in 1+\fm_2^m$, we have that $[\bar f(a)+d]_m=[\bar f(a)]_m$. Therefore, $f([a]_n)=[\bar f(a)]_m$ for each $a\in \W(k_1)$.
\end{proof}

\subsection{Tamely ramified case}\label{subsection:tamely_ramifieid}
We first recall some embedding lemma in \cite{Kuh} for tame algebraic extensions of valued fields. We say that an algebraic extension $(L,\nu_L)$ of a henselian valued field $(K,\nu_K)$ is {\em tame} if the following conditions hold: For every finite subextension $(F,\nu_F)$ of $(L,\nu_L)$ over $(K,\nu)$,
\begin{itemize}
	\item the residue field extension $k_L$ over $k_K$ is separable;
	\item if the characteristic of $k_K$ is $p>0$, then the ramification index $(\Gamma_F:\Gamma_K)$ is prime to $p$; and
	\item $[F:K]=[k_F:k_K](\Gamma_F:\Gamma_K)$.
\end{itemize}
If $K$ is a complete discrete tamely ramified valued field with a perfect residue field $k$, then $K$ is the tame extension of the fraction field of $\W(k)$.
\begin{fact}\cite[Lemma 3.1]{Kuh}\label{fact:embedding_tame}
Let $K$ be an arbitrary valued field. Let $L$ be an algebraic tame extension of some henselization of $K$ and $F$ be an arbitrary henselian extension of $K$. Any $K_{1}$-embedding from $L_{1}$ to $F_{1}$ is induced from a $K$-embedding from $L$ to $F$. Furthermore if $k_K=k_L$, then any (group) $G_K^1$-embedding from $G_L^1$ to $G_F^1$ is induced from a $K$-embedding from $L$ to $F$.
\end{fact}
\noindent By adapting the ideas of the proofs of \cite[Lemma 3.1]{Kuh}, we prove the following theorems.

\begin{theorem}\label{thm:tame_vf_vhf}
Let $K_1$ and $K_2$ be complete discrete valued fields of mixed characteristic $(0,p)$ with perfect residue fields. Suppose $K_1$ is tamely ramified. Any homomorphism over $p$ from $\H_n(K_1)$ to $\H_m(K_2)$ is induced from a unique homomorphism from $K_1$ to $K_2$. From this, we conclude that there is one-to-one correspondence between $\Hom_{\rg}(\H_n(K_1),\H_m(K_2))$ and $\Hom(K_1,K_2)$.
\end{theorem}
\begin{proof}
Take $f\in \Hom_{\rg}(\H_n(K_1),\H_m(K_2))$. Let $F_i$ be the fraction field of $\W(k_i)$ where $k_i$ is the residue field of $K_i$ for $i=1,2$. By Lemma \ref{lem:preserving_Teichmuller}(2), $f\restriction_{\H_1(F_1)}$ is induced from a homomorphism $\bar f:F_1\rightarrow F_2$. Since $K_1$ is a totally tamely ramified extension of $F_1$, $K_1=F_1(\sqrt[e]{pa})$ for some $a$ in $\W(k_1)^{\times}$ where $e$ is the ramification index of $K_1$ (c.f. Chapter $2$ of \cite{L}). Let $\pi_1=\sqrt[e]{pa}$ and $f([\pi_1]_n)=[\pi_2']_m$ for some $\pi_2'\in K_2$ so that $[\pi_2']_m^e=[p\bar f(a)]_m$. Consider a polynomial $P(X)=X^e-(p\bar f(u))/\pi_2'^e\in K_2[X]$. Since $P(1)\in \fm_2^m$ and $P'(1)\notin \fm_2$, by Hensel's lemma, there is unique $b\in R_2^{\times}$ such that $b^e=(p\bar f(a))/\pi_2'^e$ and $(b-1)\in \fm_2^m$. Let $\pi_2:= b\pi_2'$ so that $\pi_2^e=p\bar f(a)$. Note that $[\pi_2]_m=[\pi_2']_m$ and $\pi_2$ is such a unique zero of the polynomial $X^e-p\bar f(a)$. So we have a homomorphism $\tilde{f}:K_1\rightarrow K_2$ extending $\bar f\cup\{(\pi_1,\pi_2)\}$ and it induces $f$.
\end{proof}
We generalize Theorem \ref{thm:tame_vf_vhf} to the case of infinitely tamely ramified valued fields. To do this, we first recall the Ax-Sen-Tate theorem.
\begin{definition}\label{def:Ax_diameter_conjugate}\cite{A}
Let $(K,\nu)$ be a henselian valued field and $(K^{alg},\nu)$ be the algebraic closure of $K$. For $a\in K^{alg}$, define $$\Delta_K(a):=\min \{\nu(\sigma(a)-a)|\ \sigma\in G_K \},$$ where $G_K$ is the Galois group of $K^{alg}$ over $K$.
\end{definition}
\begin{fact}\label{fact:Ax_zeros_approximation}\cite[Proposition 1, Proposition 2']{A}
Let $(K,\nu)$ be a complete valued field and $(K^{alg},\nu)$ be the algebraic closure of $K$. Suppose $\Gamma_{K^{alg}}$ is archimedean. Let $K\subset F\subset K^{alg}$.
\begin{enumerate}
	\item Suppose $K$ is of mixed characteristic $(0,p)$. Then for all $a\in K^{alg}$, there exists $b\in F$ such that $$\nu(a-b)\ge \Delta_F(a)-(p/(p-1)^2)\nu(p).$$
	\item Suppose $K$ is of equal characteristic $p\ge 0$. Then for all $a\in K^{alg}$ and for all $\gamma\in (\Gamma_{K^{alg}})_{> 0}$, there exists $b$ in the perfect closure of $F$ such that $$\nu(a-b)\ge \Delta_F(a)-\gamma.$$
\end{enumerate}
\end{fact}

\begin{fact}[Ax-Sen-Tate Theorem]\label{fact:structure_complete_subfields}\cite[Proposition 3.8]{FO}
Let $(K,\nu)$ be a complete valued field and $(K^{alg},\nu)$ be the algebraic closure of $K$ with $\Gamma_{K^{alg}}$ archimedean. Let $C$ be the completion of $K^{alg}$ which is algebraically closed and let $L$ be a perfect complete subfield of $C$ containing $K$. Then $L$ is the completion of $L\cap K^{alg}$.
\end{fact}
\begin{proof}
See \cite[Proposition 3.8]{FO} with Fact \ref{fact:Ax_zeros_approximation}.
\end{proof}
\begin{corollary}\label{cor:general_tamely_ramifieid}
Let $p$ be a prime number. Let $K$ be a tamely ramified valued field of mixed characteristic $(0,p)$ with a perfect residue field $k$, and $F$ be the fraction field of $\W(k)$. Suppose $K$ is a subfield of the completion of $F^{alg}$, and either
\begin{itemize}
	\item $K$ is an algebraic extension of $F$, or
	\item complete.
\end{itemize}
For a complete valued field $L$ of mixed characteristic $(0,p)$ with a perfect residue field, any homomorphism over $p$ from $\H_\gamma(K)$ to $\H_\lambda(L)$ is induced from a unique homomorphism $K$ to $L$ for any $\gamma\in (\Gamma_{K})_{\ge 0}$ and $\lambda\in (\Gamma_{L})_{\ge 0}$.
\end{corollary}
\begin{proof}
Let $L$ be a complete valued field of mixed characteristic $(0,p)$ with a perfect residue field. Let $F$ be the fraction field of $\W(k)$. Fix $\gamma\in (\Gamma_{K})_{\ge 0}$ and $\lambda\in (\Gamma_{L})_{\ge 0}$.\\

Suppose $K$ is an algebraic extension of $F$. Since $K$ is tamely ramified, $K=\bigcup K_i$ where $K_i$ is a tamely totally ramified finite extension of $F$. Take $f\in \Hom_{\rg}(\H_{\gamma}(K),\H_{\lambda}(L))$. Set $f_i:=f\restriction_{\H_{\gamma}(K_i)}$ and we have $f:=\lim_{\longleftarrow K_i}\limits f_i=\bigcup f_i$. By Theorem \ref{thm:tame_vf_vhf}, $f_i$ is induced from a unique homomorphism $\sigma_i:K_i\rightarrow L$. Take $\sigma:=\lim_{\longleftarrow K_i}\limits \sigma_i=\bigcup \sigma_i:K\rightarrow L$ and it is a unique homomorphism inducing $f$.\\

Suppose $K$ is complete. By Fact \ref{fact:structure_complete_subfields}, $K$ is the completion of $K':=K\cap F^{alg}$. Let $f'=f\restriction_{\H_{\gamma}(K')}$. By above result, $f'$ is induced from a unique homomorphism $\tilde f':K'\rightarrow L$. Since $K$ is a completion of $K'$, $\tilde f'$ induces a unique homomorphism $\tilde{f}\rightarrow L$, which induces $f$.
\end{proof}

\noindent We can not drop the condition of being over $p$ in Theorem \ref{thm:tame_vf_vhf}.
\begin{example}\label{ex:criteria_via_vhf_2}
Consider $K_1=\BQ_3(\sqrt{3})$ and $K_2=\BQ_3(\sqrt{-3})$, which are not isomorphic by Kummer Theory. Note that their residue fields are isomorphic to $\BF_3$ and so their Teichm\"{u}ller representatives are $\{-1,0,1\}$. Let $\pi_1=\sqrt{3}$ and $\pi_2=\sqrt{-3}$. Every elements of $K_1$ is of the form $\sum_{i\ge n}a_i \pi_1^n$ for some integer $n$ and $a_i\in\{-1,0,1\}$ with $a_n\neq 0$. So each elements in $\H_1(K_1)$ is of the form $\pi_1^n a_n(1+\fm_1)$ for some integer $n$ and $a_n\in\{-1,1\}$($\dagger$). The same formulas hold for $K_2$. We will show that $\Hom(\H_1(K_1),\H_1(K_2))=\{f_0,f_1\}$, where $f_i$ sends $[a]$ to $[a]$ for $a\in\{-1,1\}$, $[\pi_1]$ to   $[(-1)^i\pi_2]$, and $[3]$ to $[-3]$ for $i=1,2$, so that $\Hom_{\rg}(\H_1(K_1),\H_1(K_2))=\emptyset$. Note that $[1]\neq [-1]$. Let $f\in \Hom(\H_1(K_1),\H_1(K_2))$. Then $f([-1])=[-1]$ because $f([1])=[1]$ and $0\in f([1]+[-1])$. Since $\pi_1$ is an uniformizer, $f([\pi_1])=[\pi_2]$ or $=[-\pi_2]$. In both cases, $f([3])=f([\pi_1]^2)=[\pi_2]^2=[\pi_2^2]=[-3]$. By ($\dagger$), such $f$ induces an isomorphism between $\H_1(K_1)$ and $\H_1(K_2)$.
\end{example}

\noindent Without base fields, even the residue fields of $K_1$ and $K_2$ are primes fields so that the residue fields are equal, we can not lift a group homomorphism from $\H_1^{\times}(K_1)$ to $\H_1^{\times}(K_2)$ to a homomorphism from $K_1$ to $K_2$(c.f. Fact \ref{fact:embedding_tame}).
\begin{example}\label{ex:nobase_nogroup}
Let $K_1=K_2=\BQ_5$. The set of Teichm\"{u}ller representatives of $\BQ_5$ is $\{0,i,i^2,i^3,i^4\}$ where $i=\sqrt{-1}$. Consider a group isomorphism $f:\H_1^{\times}(\BQ_5)\rightarrow \H_1^{\times}(\BQ_5)$ by mapping $[i]\mapsto [-i]$ and $[5]\mapsto[5]$. Then $f$ is never induced from an automorphism of $\BQ_5$ since any automorphism of $\BQ_5$ sends $i$ to $i$.
\end{example}

\subsection{Generally ramified case}\label{subsection:generally_ramified}
We first introduce of a notion of lifting map of homomorphisms of the $n$-th valued hyperfields, which is an analogy to a lifting map of homomorphisms of the $n$-th residue rings in \cite[Definition 3.1]{LL}.

\begin{definition}\label{def:lifting homomorphim}
Let $K_1$ and $K_2$ be complete discrete valuation rings of characteristic $0$ with perfect residue fields $k_1$ and $k_2$ of characteristic $p$ respectively. Let $\pi_i$ be a uniformizer of $K_i$ and $\nu_i$ be a corresponding valuation of $K_i$ for $i=1,2$. For any homomorphism $\phi:\H_n(K_1)\rightarrow \H_m(K_2)$, we say that  a homomorphism $g:K_1\rightarrow K_2$  is a \emph{$(n,m)$-lifting} of $\phi$ at $\pi_1$ if $g$ satisfies the following:
 \begin{itemize}
\item
There exists a representaive $b$ of $\phi([\pi_1])$ which satisfies
$$
  \nu_2\big(g(\pi_1)-b\big)> M(K_1).
$$
\item
$\phi_{red,1} \circ \H_1=\H_1 \circ g$ where $\phi_{red,1}:\H_1(K_1) \rightarrow \H_1(K_2) $ denotes the natural reduction map of $\phi$.
\end{itemize}
When such $g$ is unique, we denote $g$ by $\L_{\pi_1, n,m}^{\H}(\phi)$. When $\L_{\pi_1, n,m}^{\H}(\phi)$ exists for all $\phi:\H_n(K_1)\rightarrow \H_m(K_2)$, we write $\L_{\pi_1, n,m}^{\H}:\Hom_{\BZ}(\H_n(K_1),\H_m(K_2))\rightarrow \Hom(K_1,K_2)$. When $n=m$, we briefly write $\L_{\pi_1, n,m}^{\H}=\L_{\pi_1, n}^{\H}$ and say that $\L_{\pi_1, n}^{\H}$ is an \emph{$n$-lifting  at $\pi_1$}.
\end{definition}
\noindent The following result is analogous to Proposition 2.9(2) for $n$-th residue rings in \cite{LL}.
\begin{proposition}\label{prop:lifting_independent_uniformizer}
Let $K_1$ and $K_2$ be complete discrete valued fields of characteristic $0$ with perfect residue fields $k_1$ and $k_2$ of characteristic $p$ respectively.  Let $\pi_i$ be a uniformizer of $K_i$ and $\nu_i$ be a corresponding valuation of $K_i$ for $i=1,2$. Let $R_i$ be the valuation ring of $K_i$ for $i=1,2$. The definition of liftings is independent of the choice of uniformizer of $K_1$. More precisely, saying that $g:K_1\rightarrow K_2$ is a $(n,m)$-lifting of $\phi:\H_n(K_1)\rightarrow \H_m(K_2)$ at $\pi_1$ is equivalent to the following:
\begin{enumerate}
\item
For any $x$ in $R_1$, there exists a representative $b_x$ of $\phi(x(1+\fm_1^n))$ which satisfies
$$
  \nu_2\big(g(x)-b_x\big)
  > M(K_1).
  $$
\item
$\phi_{red,1} \circ \H_1=\H_1 \circ g$
\end{enumerate}
We write $\L_{\pi_1, n, m}^{\H}=\L_{n, m}^{\H}$ and say that $\L_{n, m}^{\H}$ is a \emph{$(n,m)$-lifting}. Moreover, there is at most one $(n,m)$-lifting for $m>M(K_1)e_2$, where $e_2$ is the ramification index of $K_2$.
\end{proposition}
\begin{proof}
The proof is similar to the proof of \cite[Proposition 3.5 (2)]{LL}. Let $S_i$ be the Teichm\"{u}ller representatives of $K_i$ for $i=1,2$. Fix a uniformizer $\pi$ of $K_1$. Let $g:K_1\rightarrow K_2$ be a $(n,m)$-lifting of $\phi:\H_n(K_1)\rightarrow \H_m(K_2)$ at a uniformizer $\pi$. Let $b \in K_2$  be a representative of $\phi([\pi]))$ such that $\nu_2(g(x)-b)>M(K_1)$. Note that $\nu_2(b)>0$. Take $x\in R_1$. Then $x=\sum_{i\ge 0}\lambda_i \pi^i$. Take $l>0$ such that $\nu_1(\sum_{i>l}\lambda_i\pi^i)>M(K_1)$. Denote $x^{\le l}:=\sum_{i\le l}\lambda_i \pi^i$ and $x^{>l}\sum_{i>l}\lambda_i \pi^i$ so that $x=x^{\le l}+x^{>l}$. Then we have that
\begin{align*}
\phi([x])&=[x^{\le l}+x^{>l}])\\
&\in \big ( \phi([x^{\le l}]+_{\H}\phi([x^{>l}]) \big )\\
&\subset \big( \sum_{i\le l}^{\H} \phi([\lambda_i])\phi([\pi]^i)+\phi([x^{>l}])  \big)\\
&\subset \big( \sum_{i\le l}^{\H} [g(\lambda_i)][b]^i+\phi([x^{>l}])  \big).
\end{align*}
There is a representative $b_x$ of $\phi([x])$ of the form:$$b_x=\sum_{i\le l}g(\lambda_i)b^i +d$$ for some $\nu_2(d)>M(K_1)$. Compute
\begin{align*}
\nu_2(g(x)-b_x)&=\nu_2\big( \sum_i (g(\lambda_i)g(\pi)^i)- b_x \big)\\
&=\nu_2\big(\sum_{0<i\le l}  g(\lambda_i)(g(\pi)^i-b^i)+(\sum_{i>l} g(\lambda_i)g(\pi)^i-d) \big)\\
&\ge \min\{\nu_2(\sum_{0<i\le l}  g(\lambda)(g(\pi)^i-b^i)),\nu_2(\sum_{i>l} g(\lambda_i)g(\pi)^i-d)\}\\
&>M(K_1)
\end{align*}
because $\nu_2(g(\pi)^i-\beta^i)=\nu_2(g(\pi)-\beta)+\nu_2(g(\pi)^{i-1}+\cdots + \beta^{i-1})>M(K_1)$, and $\nu_2(\sum_{i>l} g(\lambda_i)g(\pi)^i),\nu_2(d)>M(K_1)$.\\

Now we show moreover part. Assume $m>M(K_1)e_2$ and there are two $(n,m)$-liftings $\L,\L': \Hom_{\BZ}(\H_n(K_1),\H_m(K_2))\rightarrow \Hom(K_1,K_2)$. Fix $f\in \Hom_{\rg}(\H_n(K_1),\H_m(K_2))$. By Lemma \ref{lem:preserving_Teichmuller} and $(2)$, we have that $\L(f)\restriction \W(k_1)=\L'(f)\restriction \W(k_1)$ $(\dagger)$. It remains to show that $\L(f)(\pi_1)=\L'(f)(\pi_1)$. Set $\pi:=\L(f)(\pi_1)$ and $\pi':=\L'(f)(\pi_1)$. By $(\dagger)$, $\pi$ and $\pi'$ are conjugates over the fraction field of $\W(k_2)$ $(\ddagger)$. By $(1)$, there are two representatives $b$ and $b'$ of $f([\pi_1])$ such that $$\nu_2(\pi-\beta),\nu_2(\pi'-\beta')>M(K_1).$$ Since $\nu_2(b-b')>m/e_2>M(K_1)$, we have that
\begin{align*}
\nu_2(\pi-\pi')&>\max\{\nu_2(\pi-b),\nu_2(b-b'),\nu_2(\pi'-b')\}\\
&>M(K_1).
\end{align*}
By $(\ddagger)$, we conclude that $\nu_2(\pi-\pi')=\infty$ and $\pi=\pi'$.
\end{proof}

\begin{remark}
Let $K_1$ and $K_2$ be complete discrete valued field of mixed characteristic $(0,p)$ with perfect residue fields. Suppose $K_1$ is tamely ramified. By Theorem \ref{thm:tame_vf_vhf}, for every $n\ge 1$ there is a unique bijective $n$-lifting map  
$$\L_n^{\H}:\Hom_{\BZ}(\H_n(K_1),\H_n(K_2))\rightarrow \Hom(K_1,K_2)$$ such that $f([a]_n)=[\L_n^{\H}(f)(a)]$ for $f\in \Hom_{\BZ}(\H_n(K_1),\H_n(K_2))$ and $a\in K_1$.
\end{remark}

\begin{fact}[Krasner's lemma]\label{lem:Krasner}
Let $(K,\nu)$ be henseilan  valued field whose value group is contained in $\BR$ and let $a,b\in K^{alg}$. Suppose $a$ is separable over $K(b)$. Suppose that for all embeddings $\sigma(\neq id)$ of $K(a)$ over $K$, we have
$$\wi{\nu}(b-a)>\wi{\nu}\big(\sigma (a)-a\big).
$$
 Then $K(a)\subset K(b)$.
\end{fact}
\noindent We show that a lifting map of homomorphisms of $n$-th valued hyperfields for large enough $n$ 
\begin{theorem}\label{thm:mainhomlifting}
Let $K_1$ and $K_2$ be complete discrete valued field of mixed characteristic $(0,p)$ with perfect residue fields. Let $e_1$ and $e_2$ be ramification indices of $K_1$ and $K_2$ respectively. Let $\H_n(K_1)$ and $\H_m(K_2)$ be valued hyper fields of $K_1$ and $K_2$ respectively. Suppose $m>M(K_1)e_1e_2$. There is a unique lifting map $\L_{n,m}^{\H}:\Hom_{\rg}(\H_n(K_1),\H_m(K_2))\rightarrow \Hom(K_1,K_2)$.
\end{theorem}
\begin{proof}
Fix a homomorphism $f$ in $\Hom_{\rg}(\H_n(K_1),\H_m(K_2))$. By Lemma \ref{lem:preserving_Teichmuller}, $f$ induces a map $f\restriction_{\H_n(S_1)}$ from $\H_n(S_1)$ to $\H_m(S_2)$. This map $f\restriction_{\H_n(S_1)}$ induces a homomorphism from $k_1$ to $k_2$ and by the functoriality of Witt ring, we have a homomorphism $\bar f:\W(k_1)\rightarrow \W(k_2)$. Note that $\bar f$ induces a homomorphism from $\W(k_1)[X]$ to $\W(k_2)[X]$ by acting on coefficients and we denote this homomorphism by $\bar f$ also. By Fact \ref{fact:witt_teichmuller}(2), there is a uniformizer $\pi_1$ of $K_1$ such that $R_1=\W(k_1)[\pi_1]$. Let $q(X)$ be the irreducible polynomial of $\pi_1$ over the fraction field of $\W(k_1)$, which is in $\W(k_1)[X]$. Write $q(X)=X^e+a_{e-1}X^{e-1}+\ldots+a_0$ where $e=\nu_1(p)$ is the ramification index of $K_1$. Let $\pi_2'\in R_2$ such that $[\pi_2']_m=f([\pi_1]_n)$. For $a\in \W(k_1)$, we can uniquely write $a=a_0+a_1p+a_2p^2+\ldots$ with $a_i\in S_1$. For $l\ge 0$, define $a^{\le l}:=a_0+a_1p+\ldots+a_lp^l$ and $a^{>l}:=a-a^{\le l}$. And define $q^{\le l}(X):=X^e+a_{e-1}^{\le l}X^{e-1}+\ldots+a_0^{\le l}$ and $q^{>l}(X)=q(X)-q^{\le l}(X)$, which are in $\W(k_1)[X]$. Then we have that
\begin{align*}
0&=f([q(\pi_1)]_n)\\
&=f([q^{\le m}(\pi_1)+q^{>m}(\pi_1)]_n)\\
&\in f([q^{\le m}(\pi_1)]_n)\hp f([q^{> m}(\pi_1)]_n)\\
&\subset [\pi_2'^e]_m\hp f[a_{e-1}^{\le m}\pi_2'^{e-1}]_m\hp\ldots\hp f[a_0^{\le m}]_m \hp f([q^{>m}(\pi_1)]_n)\\
&\subset [\pi_2'^e]_m\hp\hsum_{0\le i<e}\limits \hsum_{0\le j\le m}\limits [\pi_2'^i\bar f(a_i^j)p^j]_m\hp f([q^{>m}(\pi_1)]_n).
\end{align*}
Note that $\nu_2(q^{>m}(\pi_1))\ge m/e_2$. By Lemma \ref{lem:basic_on_hyperfield}(2), we have that $$0=\bar f(q^{\le m})(\pi_2')+d (*)$$ for some $d\in \fm_2^m$. Since $q=q^{\le m}+q^{>m}$, we have that $\bar f(q)(\pi_2')=\bar f(q^{\le m})(\pi_2')+\bar f(q^{> m})(\pi_2')$ and $\nu_2(\bar f(q^{> m})(\pi_2'))\ge m/e_2$. From $(*)$ we have that $$0=\bar f(q)(\pi_2')+d'$$ for some $d'\in \fm_2^m$, and $\nu_2(\bar f(q)(\pi_2'))\ge m/e_2>M(K_1)e_1$. By Krasner's lemma, there is $\pi_2\in K_2$ such that $\bar f(q)(\pi_2)=0$ with $\nu_2(\pi_2-\pi_2')>M(K_1)$. Define $\tilde f:K_1\rightarrow K_2,\pi_1\mapsto \pi_2$ extending $\bar f$, and set $\L_{n,m}^{\H}(f)=\tilde f$.
\end{proof}

\begin{corollary}\label{cor:criteria_isom}
Let $(K_1,\nu_1,k_1,\Gamma_1)$ and $(K_2,\nu_2,k_2,\Gamma_1)$ be finitely ramified complete valued fields of mixed characteristic. Let $n> e_{\nu_2}(p)(1+e_{\nu_1}^2(p))$. The following are equivalent:
		\begin{enumerate}
			\item $K_1\cong K_2$;
			\item $R_{n}(K_1)\cong R_{n}(K_2)$; and
			\item $\H_{n}(K_1)\cong_{\H_n(\{p\})} \H_{n}(K)$.

		\end{enumerate}
\end{corollary}

The following is an analogy of \cite[Proposition 4.4]{LL} for a lifting map of homomorphisms of $n$-th valued hyperfields, which gives a funtoriality of lifting map in Section \ref{section:functoriality}.

\begin{proposition}\label{prop:funtoriality}
Let $(K_1,\nu_1)$, $(K_2,\nu_2)$, and $(K_3,\nu_3)$ be complete discrete valued fields of mixed characteristic $(0,p)$ with perfect residue fields. Suppose $K_1$ and $K_2$ have the same ramification index. Suppose $m,k> \max\{ e_{\nu_2}(p)(1+e_{\nu_1}^2(p), e_{\nu_3}(p)(1+e_{\nu_2}^2(p))\} $. Let $f\in \Hom_{\rg}(\H_n(K_1),\H_m(K_2))$ and $g\in \Hom_{\rg}(\H_m(K_2),\H_k(K_3))$. Then $\L_{n,k}^{\H}(g\circ f)=\L_{m,k}^{\H}(g)\circ \L_{n,m}^{\H}(f)$. 
\end{proposition}
\begin{proof}
Take $f\in \Hom_{\rg}(\H_n(K_1),\H_m(K_2))$ and $g\in \Hom_{\rg}(\H_m(K_2),\H_k(K_3))$. First, we have that
\begin{align*}
(g\circ f)\circ \H_1 &= g\circ (f\circ \H_1)\\
&= g\circ (\H_1\circ \L_{n,m}^{\H}(f))\\
&= (g\circ \H_1)\circ \L_{n,m}^{\H}(f)\\
&= (\H_1\circ \L_{m,k}^{\H}(g))\circ \L_{n,m}^{\H}(f)\\
&= \H_1\circ(\L_{m,k}^{\H}(g)\circ \L_{n,m}^{\H}(f)).
\end{align*}
It remains to show that for a uniformizer $\pi_1$ of $K_1$, there is a representative $b$ of $g\circ f([\pi_1])$ such that $$\nu_3(\L_{m,k}^{\H}(g)\circ \L_{n,m}^{\H}(f)(\pi_1)-b)>M(K_1).$$ Take a representative $b_1$ of $f([\pi_1])$ such that $\nu_2(\L_{n,m}^{\H}(f)(\pi_1)-b_1)>M(K_1)$, and we have that $\nu_3(\L_{m,k}^{\H}(g)\circ \L_{n,m}^{\H}(f)(\pi_1)-\L_{m,k}^{\H}(g)(b_1))>M(K_1)$. Next consider a representative $b_2$ of $g([b_1])( =g\left (f([\pi_1])\right) )$ such that $\nu_3(\L_{m,k}^{\H}(g)(b_1)-b_2)>M(K_2)$. Since $K_1$ and $K_2$ have the same ramification indices, we have that $M(K_1)=M(K_2)$. Thus we conclude that $\nu_3(\L_{m,k}^{\H}(g)\circ \L_{n,m}^{\H}(f)(\pi_1)-b_2)>M(K_1)$, and set $b:=b_2$.
\end{proof}

\section{Valued hyperfields, truncated DVRs, and valued fields}\label{section:functoriality}
In this section, we figure out relationships between valued hyperfields, truncated DVRs, and complete discrete valued fields of mixed characteristic. We first define a notion of finitely ramified valued hyperfields.

\begin{definition}\label{def:char_ramificationindex_vhf}
Let $(H,\nu)$ be a discrete valued hyperfield of length $l$. Let $R_l(H)=R$. Suppose the residue field $k(H)$ has the characteristic $p\ge 0$.
\begin{enumerate}
	\item We say $H$ is of {\em equal characteristic $(p,p)$} if $p=0$ in $R$. Otherwise, we say $H$ is of {\em mixed characteristic $(0,p)$}.
	\item Suppose $H$ is {\em of mixed characteristic $(0,p)$}. We say $H$ is {\em finitely ramified} if $\nu_R(p)=e<\infty$, and $e$ is called the {\em ramification index} of $H$.
	\item Suppose $H$ is finitely ramified. We say $H$ is normalized if $\theta_H=1/e$ where $e$ is the ramification index of $H$.
\end{enumerate}
\end{definition}
\noindent Note that any discrete valued hyperfield of mixed characteristic need not be finitely ramified.

\begin{proposition}\label{prop:DVR_enough}
Let $K$ be a discrete valued field. Let $K_1$, $K_2$, and $K_3$ be discrete valued fields of mixed characteristic $(0,p)$ having the same ramification index $e$.
\begin{enumerate}
	\item $T_n(K)=(R_n(K),M_n(K),\epsilon_n)$ is a triple, where $M_n(K):=\fm_K/\fm_K^{n+1}$ and the map $\epsilon_n$ is induced by the inclusion $\fm_K\subset R(K)$.
	\item For every triple $T=(R_n(K),M,\epsilon)$, we have $T\cong T_n(K)$.
	\item Let $n>e$. Each morphism $f\in \Hom(R_n(K_1), R_n(K_2))$ induces a morphism $T_n(f)\in \Hom(T_n(K_1),T_n(K_2))$. For $f\in \Hom(R_n(K_1),R_n(K_2))$ and $g\in \Hom(R_n(K_2),R_n(K_3))$, we have $T_{n}(g\circ f)=T_{n}(g)\circ T_{n}(f)$.
\end{enumerate}
\end{proposition}
\begin{proof}
(1) It is clear.\\

(2) Let $R:=R_n(K)$. Given a triple $T=(R,M,\epsilon)$, let $\Pi$ be a generator of $M$. Since $\epsilon(M)=\fm_R$, $\epsilon(\Pi)=\pi+\fm_K$ for a uniformizer $\pi$ of $K$. Define a map $\eta:M\rightarrow \fm_K/\fm_K^{n+1},\ \Pi\mapsto \pi/\fm_K^{n+1}$. Then the triple $(1,\id_R,\eta)$ gives an isomorphism from $T$ to $T_n(K)$.\\

(3) Let $f\in \Hom(R_n(K_1),R_n(K_2))$. Let $\pi_1$ and $\pi_2$ be uniformizers of $K_1$ and $K_2$ respectively. By the choice of $n$, $f(\pi_1/\fm_{K_1}^n)$ generates the maximal ideal of $R_n(K_2)$ so that $f(\pi_1/\fm_{K_1}^n)=(a_f\pi_2)/\fm_{K_2}^{n+1}$ for a unit $a_f\in R(K_2)^{\times}$. Define a map $\eta_f:M_n(K_1)\rightarrow M_n(K_2)$ sending $\pi_1/\fm_{K_1}^{n+1}$ to $(a_f\pi_2)/\fm_{K_2}^{n+1}$. Then the triple $T_n(f):=(1,f,\eta_f)$ gives a morphism from $T_n(K_1)$ to $T_n(K_2)$. Note that $T_n(f)$ does not depend on the choices of $\pi_1$ and $\pi_2$ because $M_n(K_1)$ and $M_n(K_2)$ are free $R_n(K_1)$ and $R_n(K_2)$-modules of rank $1$ respectively.

Now we show $T_n$ is commute with the composition. Let $\pi_3$ be a uniformizer of $K_3$. Fix $f\in \Hom(R_n(K_1),R_n(K_2))$ and $g\in \Hom(R_n(K_2),R_n(K_3))$. There are $a_f\in R(K_2)^{\times}$, $b_g, b_{g\circ f}\in R(K_3)^{\times}$ such that $f(\pi_1/\fm_{K_1}^n)=(a_f\pi_2)/\fm_{K_2}^n$, $g(\pi_2/\fm_{K_1}^n)=(b_g\pi_3)/\fm_{K_3}^n$, and $g\circ f(\pi_1/\fm_{K_1}^n)=(b_{g\circ f}\pi_3)/\fm_{K_3}^n$. Let $b_f\in R(K_3)^{\times}$ such that $g(a_f/\fm_{K_2}^n)=b_f/\fm_{K_3}^n$. Then we have that $b_{g\circ f}/\fm_{K_3}^n=(b_gb_f)/\fm_{K_3}^n$. By choosing $b_g$ and $b_f$ properly, we may assume that $b_gb_f-b_{g\circ f}\in \fm_{K_3}^{n+1}$ $(\dagger)$. Define $\eta_{g\circ f}: M_n(K_1)\rightarrow M_n(K_3), \pi_1/\fm_{K_1}^{n+1}\mapsto (b_{g\circ f}\pi_3)/\fm_{K_3}^{n+1}$. By $(\dagger)$, we have that $\eta_{g\circ f}=\eta_g\circ \eta_f$. Therefore, \begin{align*}
T_n(g\circ f)&= (1,g\circ f, \eta_{g\circ f})\\
&= (1,g\circ f, \eta_g\circ \eta_f)\\
&= (1,g,\eta_g)\circ (1,f,\eta_f)\\
&= T_n(g)\circ T_n(f).
\end{align*}
\end{proof}

\begin{corollary}\label{cor:vhf_unique_realizable}
Let $H$ be a finitely ramified discrete valued hyperfield of mixed characteristic with the perfect residue field and let $e$ be the ramification index. Suppose $l(:=l(H))>e(1+\nu_R(e))$ so that $e$ is not zero in $R(:=R_l(H))$. Then there is a unique complete discrete valued field $K$ (up to isomorphic) such that $\H_l(K)\cong H$. 
\end{corollary}

\begin{rem/def}\label{rem/def:general_HomoverZ}
Let $H_1$ and $H_2$ be finitely ramified discrete valued hyperfield of mixed characteristic $(0,p)$ with perfect residue fields. Let $\Tr(H_1)=(R_1,M_1,\epsilon_1)$ and $\Tr(H_2)=(R_2,M_2,\epsilon_2)$, and we identify $H_1=\U(\Tr(H_1))$ and $H_2=\U(\Tr(H_2))$. Suppose they have the same length $l$ and the same ramification index $e$ so that $n=\nu_1(e)=\nu_2(e)\in \{0,1,\ldots, l-1\}\cup\{\infty\}$, where $\nu_1=\nu_{R_1}$ and $\nu_2=\nu_{R_2}$. Suppose $l>e$. Since $l>e$, $p$ is not zero in $R_l$ and $R_2$ and $\nu_1(p)=\nu_2(p)=e$. So $p=a_1\pi_1^e$ in $R_1$ and $p=a_2\pi_2^e$ in $R_2$ for some units $a_1,a_2$ and some uniformizers $\pi_1,\pi_2$ in $R_1$ and $R_2$ respectively. Take $\Pi_i\in M_i$ such that $\epsilon_i(\Pi_i)=\pi_i$ for $i=1,2$. Then $\nu_{\Tr(H_i)}(a_i \Pi_i^{\otimes e})=e$ for $i=1,2$. We say a homomorphism $f:H_1\rightarrow H_2$ is {\em over $p$} if $f(a_1\Pi_1^{\otimes e})=a_2\Pi_2^{\otimes e}$. Denote $\Hom_{\rg}(H_1,H_2)$ for the set of all homomorphisms from $H_1$ to $H_2$, which are over $p$. 
\end{rem/def}

\begin{example}\label{ex:Hom_and_HomoverZ}
Let $K_1=K_2=\BQ_3(\sqrt{3})$ and $\fm$ be the maximal ideal of the valuation ring $\BZ_3[\sqrt{3}]$. Note that $R:=R_4(K_1)=R_4(K_2)\cong (\BZ_3/9\BZ_3)[x]/(x^2-3)$. Then there is an isomorphism $f:a+bx\mapsto a+4bx$ in $\Hom(R_4(K_1),R_4(K_2))$. Then $f$ induces an isomorphism $T_4(f):(R,\fm/\fm^4,\epsilon_4)\rightarrow (R,\fm/\fm^4,\epsilon_4)$ and it induces an isometric isomorphism $U(T_4(f)):\H_4(K_1)\rightarrow \H_4(K_2),\sqrt{3}(1+\fm^4)\mapsto4\sqrt{3}(1+\fm^4)$, which is not in $\Hom_{\rg}(\H_4(K_1),\H_4(K_2))$. Suppose $U(T_4(f))$ is over $p$. Then $3(1+\fm^4)=U(T_4(f))(3(1+\fm^4))=U(T_4(f))((\sqrt{3})^2(1+\fm^4))=(4\sqrt{3})^2(1+\fm^4)=4^2 3(1+\fm^4)$. So, we have $(1+\fm^4)=4^2 (1+\fm^4)$, which is impossible, because $15\notin \fm^4=9\BZ_3$.
\end{example}

\bigskip

Now we introduce some categories of valued hyperfields, truncated DVRs, and valuation rings and we study relationships between them. We recall two categories of truncated DVRs and valuation rings, which were used to generalize the functoriality of unramified valuation rings in \cite[Section 4]{LL}. For a prime number $p$ and a positive integer $e$, let $\CC_{p,e}$ be a category consisting of the following data :
\begin{itemize}
	\item $\ob(\CC_{p,e})$ is the family of complete discrete valuation rings of mixed characteristic having perfect residue fields of characteristic $p$ and the ramification index $e$; and
	\item $\mor_{\CC_{p,e}}(R_1,R_2):=\Hom(R_1,R_2)$ for $R_1$ and $R_2$ in $\ob(\CC_{p,e})$.
\end{itemize}
Let $\CR_{p,e}^{n}$ be a category consisting of the following data :
\begin{itemize}
	\item For $n\le e$, $\ob(\CR_{p,e}^{n})$ is the family of truncated DVRs $\overline{R}$ of length $n$ with perfect residue fields of characteristic $p$, and for $n>e$, $\ob(\CR_{p,e}^{n})$ is the family of truncated DVRs $\overline{R}$ of length $n$ with perfect residue fields of characteristic $p$ such that $p\in \overline{\fm}^{e}\setminus\overline{\fm}^{e+1}$ where $\overline{\fm}$ is the maximal ideal of $\ov{R}$; and
	\item $\mor_{\CR_{p,e}^{n}}(\overline{R_1},\overline{R_2}) :=\Hom(\overline{R_1},\overline{R_2})$ for $\overline{R_1}$ and $\ov{R_2}$ in $\ob(\CR_{p,e,}^{n})$,
\end{itemize}
Note that for $e_1,e_2\ge 1$ and for $n\le e_1,e_2$, two categories $\CR_{p,e_1}^n$, $\CR_{p,e_2}^n$ are the same. For each $m>n$, let $\Pr_n:\ \CC_{p,e}\rightarrow \CR_{p,e}^n$ and $\Pr^m_n:\ \CR_{p,e}^m\rightarrow \CR_{p,e}^n$ be the canonical projection functors respectively. Given a prime number $p$ and a positive integer $e$, let $l_{p,e}:=e(1+\nu(e))$ for some(every) $R\in \ob(\CC_{p,e})$. Note that for $p\not| e$, we have that $l_{p,e}=e$.
\begin{fact}\cite[Definition 4.2, Theorem 4.7]{LL}\label{fact:lifting_DVR_valuedfield}
Fix a prime number $p$ and a positive integer $e$. For every $n>l_{p,e}$, there is a functor $\L:\CR_{p,e}^n \longrightarrow \CC_{p,e}$, called an {\em $n$-th lifting functor}, which satisfies the following:
 \begin{enumerate}
 \item
 $(\Pr_n\circ \L)(\ov{R})\cong \ov{R}$ for each $\overline{R}$ in $\ob(\CR_{p,e}^n)$.
  \item
 $\Pr_1 \circ\L$ is equivalent to $\Pr^{n}_1$.
 \item
 $\L\circ \Pr_n$ is equivalent to  $\id_{\CC_{p,e}}$.
 \end{enumerate}
Moreover, there is a unique $n$-th lifting functor $\L$ satisfying
\begin{enumerate}
\item[(4)] For each $g\in \mor_{\CR_{p,e}^n}(\overline{R}_1,\overline{R}_2)$ and for any $x\in \L(\overline{R}_1)$, there is a representative $b_x$ of $(\Pr_n\circ \L)(g)(x+\fm_1^n)$ such that $$\nu_2(\L(g)(x)-b_x)>M(\L(\overline{R}_1)).$$
\end{enumerate}
\end{fact}

Next we introduce two categories of valued hyperfields. Let $\H_{p,e}^n$ be a category consisting of the following data :
\begin{itemize}
	\item $\ob(\H_{p,e}^n)$ is the family of discrete normalized valued hyperfield of length $n$ and mixed characteristic $(0,p)$ having perfect residue fields, and for $n>e$ in addition, having ramification indices $e$; and
	\item For $n\le e$, $\mor_{\H_{p,e}^n}(H_1,H_2):=\Iso(H_1,H_2)$ and for $n>e$, $\mor_{\H_{p,e}^n}(H_1,H_2):=\Iso_{\BZ}(H_1,H_2)$, where $\Iso_{\BZ}(H_1,H_2)=\Hom_{\rg}(H_1,H_2)\cap \Iso(H_1,H_2)$ for $H_1$ and $H_2$ in $\ob(\H_{p,e}^n)$.
\end{itemize}
Let $\widehat \H_{p,e}^n$ be a category consisting of the following data :
\begin{itemize}
	\item $\ob(\widehat \H_{p,e}^n)=\ob(\H_{p,e}^n)$; and
	\item $\mor_{\widehat \H_{p,e}^n}(H_1,H_2):=\Iso(H_1,H_2)$ for $H_1$ and $H_2$ in $\ob(\widehat \H_{p,e}^n)$.
\end{itemize}
By Remark/Definition \ref{rem/def:residuefield_vhf}, we have functors $\res_{p,e}^n:\H_{p,e}^n\rightarrow \CR_{p,e}^1$ and $\wres:\wH_{p,e}^n\rightarrow \CR_{p,e}^1$. Next we introduce lifting functors for two categories $\H_{p,e}^n$ and $\wH_{p,e}^n$, which are analogous to the lifting functors of $\CR_{p,e}^n$ in \cite[Definition 4.2]{LL}.  
\begin{definition}\label{n-liftable}
Fix a prime number $p$ and a positive integer $e$. We say that the category $\CC_{p,e}$ is \emph {$n$-$\H$-liftable}  if there is a  functor $\L^{\H}:\H_{p,e}^n \rightarrow \CC_{p,e}$ which satisfies the following:
 \begin{itemize}
 \item
 $(\H_n\circ \L^{\H})(\ov{R})\cong \ov{R}$ for each $\overline{R}$ in $\ob(\CR_{p,e}^n)$.
  \item
 $\Pr_1 \circ\L^{\H}$ is equivalent to $\res_{p,e}^n$.
 \item
 $\L^{\H}\circ \H_n$ is equivalent to  $\id_{\CC_{p,e}}$.
\end{itemize}
 We say that $\L^{\H}$ is an \emph{$n$-th $\H$-lifting functor} of $\CC_{p,e}$.
\end{definition}

Now we see relationships between the categories $\CC_{p,e}$, $\CR_{p,e}^n$, $\H_{p,e}^n$, and $\wH_{p,e}^n$. Define a functor $\wU: \CR_{p,e}^n\rightarrow \wH_{p,e}^n$ as follows:
\begin{itemize}
	\item For $\overline{R}\in \ob(\CR_{p,e}^n)$, $\wU(\overline{R}):=(\U\circ T_n\circ \L)(\overline{R} )$ after rescaling $\theta_{(\U\circ T_n\circ \L)(\overline{R} )}=1/e$, where $\L$ is the $n$-th lifting in Fact \ref{fact:lifting_DVR_valuedfield}; and
	\item For $f\in \mor_{\CR_{p,e}^n}(\overline{R}_1,\overline{R}_2)$, $\wU(f):=(\U\circ \T_n)(f)$,
\end{itemize}
and define a functor $\wTo:\wH_{p,e}^n\rightarrow \CR_{p,e}^n$ as follows:
\begin{itemize}
	\item For $H\in \ob(\wH_{p,e}^n)$, $\wTo(H):=\overline{R}$ where $\overline{R}$ is a truncated DVR in $\Tr(H)$.
	\item For $g\in \mor_{\wH_{p,e}^n}(H_1,H_2)$, $\wTo(g)$ is a morphism in $ \mor_{\CR_{p,e}^n}(\wTo(H_1),\wTo(H_1))$ such that $\Tr(f)=(1,\wTo(g),\eta)$.
\end{itemize}
By Remark \ref{rem:triple_to_vhf} and Proposition \ref{prop:DVR_enough}, we have the following result.
\begin{theorem}\label{thm:equiv_DVR_hvf}
Let $p$ be a prime number and $e$ be a positive integer. Fix $n>e$.
\begin{enumerate}
	\item $\wU\circ \wTo$ is equivalent to $\id_{\wH_{p,e}^n}$.
	\item $\wTo\circ \wU$ is equivalent to $\id_{\CR_{p,e}^n}$
\end{enumerate}
Therefore, two categories $\wH_{p,e}^n$ and $\CR_{p,e}^n$ are equivalent.
\end{theorem}
\begin{remark}\cite[Proposition 4.9]{LL}\label{prop:n_atleast_e}
Let $R_1/\W(k)$ and $R_2/\W(k)$ be totally ramified extensions of degree $e$. Then $R_{1,e}$ is isomorphic to $R_{2,e}$ as $\W(k)$-algebras. So, the assumption that $n>e$ is natural in Theorem \ref{thm:equiv_DVR_hvf}.
\end{remark}

\noindent Each $R\in \ob(\CC_{p,e})$ gives a discrete valued hyperfield $\H(R):=\H_n(K)$ in $\ob(\H_{p,e}^n)$ where $K$ is the fraction field of $R$ after rescaling $\theta_{\H_n(R)}=1/e$. Also each $f\in \mor_{\CC_{p,e}}(R_1,R_2)$ induces a morphism $\H(f):\H(R_1)\rightarrow \H(R_2)$. So we have a functor $\H:\CC_{p,e}\rightarrow \H_{p,e}^n$. For $n>l_{p,e}$, we define a functor $\L^{\H}:\H_{p,e}^n\rightarrow \CC_{p,e}$.  Since $\H_{p,e}^n$ is a subcategory of $\wH_{p,e}^n$, we have a functor $\To:\H_{p,e}^n\rightarrow \CR_{p,e}^n$ by restricting $\wTo$ to $\H_{p,e}^n$. For each $H\in \ob(\H_{p,e}^n)$, there is a unique( up to isomorphic) complete valued field $K$ such that $H\cong \H_n(K)$, where $K$ is the fraction field of $(\L\circ\To)(H)$ by Corollary \ref{cor:criteria_isom}. For $H_1,H_2\in \ob(\H_{p,e}^n)$ and $f\in \mor_{\H_{p,e}^n}(H_1,H_2)$, we have a morphism $\L^{\H}(f)\in \mor_{\CC_{p,e}}(R_1,R_2)$ where $R_1=(\L\circ \To)(H_1)$ and $R_2=(\L\circ \To)(H_2)$ by Theorem \ref{thm:mainhomlifting}. By Proposition \ref{prop:funtoriality}, $\L^{\H}$ forms a functor. By Remark \ref{rem:vf_rigid_hom}, Theorem \ref{thm:mainhomlifting}, and Proposition \ref{prop:funtoriality}, we have the following result which is analogous to Fact \ref{fact:lifting_DVR_valuedfield} in the case of valued hyperfields.
\begin{theorem}\label{thm:lifting_vhf_valuedfield}
Let $p$ be a prime number and $e$ be a positive integer. Fix $n>l_{p,e}$. Then $\CH_{p,e}^n$ is $n$-$\H$-liftable and there is an $n$-th $\H$-lifting functor $\L^{\H}$ of $\CH_{p,e}^n$ satisfying the following:
 \begin{enumerate}
 \item
 There is an isometric isomorphism, which is over $p$, between $(\H\circ \L^{\H})(H)$ and $H$ for each $H$ in $\ob(\CH_{p,e}^n)$.
 \item $\Pr_1\circ\L^{\H}$ is equivalent to $\res_{p,e}^n$.
 \item
 $\L^{\H}\circ \H$ is equivalent to  $\id_{\CC_{p,e}}$.
\end{enumerate}
Moreover, there is a unique $n$-th lifting functor $\L^{\H}$ satisfying
\begin{enumerate}
\item[(4)] For each $g\in \mor_{\H_{p,e}^n}(H_1,H_2)$ and for any $x\in \L^{\H}(H_1)$, there is a representative $b_x$ of $(\H_n\circ \L)(g)([x])$ such that $$\nu_2(\L^{\H}(g)(x)-b_x)>M(\L^{\H}(H_1)).$$
\end{enumerate}
If we take $\L$ satisfying (1)-(4) and $\L^{\H}$ satisfying (1)-(4), then $\L^{\H}=\L\circ \To$.
\end{theorem}
\noindent Moreover, by Theorem \ref{thm:tame_vf_vhf}, we have the following theorem.
\begin{theorem}\label{thm:tame_equiv_vhf_valuedfields}
Let $p$ be a prime number and $e$ be a positive integer. Suppose $p$ does not divide $e$ so that $l_{p,e}=e$. For any $n>e$, two categories $\H_{p,e}^n$ and $\CC_{p,e}$ are equivalent
\end{theorem}
\noindent In summary, for $n> l_{p,e}$, we have the following diagram between categories of $\CC_{p,e}$, $\CR_{p,e}^n$, $\H_{p,e}^n$, and $\wH_{p,e}^n$:
$$
\begin{tikzcd}
\H_{p,e}^n \arrow[d, hook] \arrow[dr, "\To" description] \arrow[r, shift left, "\L^{\H}"] \arrow[r, shift right, leftarrow, "\H"'] & \CC_{p,e} \arrow[d, shift left, "\Pr_n"] \arrow[d, shift right, leftarrow, "\L"'] \\
\wH_{p,e}^n \arrow[r, shift left, leftarrow, "\wU"] \arrow[r, shift right, "\wTo"']
& \CR_{p,e}^n
\end{tikzcd}
$$
, and we can take $\L^{\H}=\L\circ \To$.

\begin{question}\label{question:equiv_vhf_valuedfields}
Let $p$ be a prime number and $e$ be a positive integer. Fix $n>l_{p,e}$. We know that two categories $\wH_{p,e}^n$ and $\CR_{p,e}^n$ are equivalent, and $\L^{\H}\circ \H$ is equivalent to $\id_{\CC_{p,e}}$. And for the tame case, that is, $p\not|e$, $\H_{p,e}^n$ and $\CC_{p,e}$ are equivalent.  Is $\H_n\circ \L^{\H}$ equivalent to $\id_{\H_{p,e}^n}$ if $p|e$?
\end{question}

\section{Relative completeness via hyper fields}\label{section:AKE_hyperfields}
In this section, we aim to prove an AKE-type relative completeness theorem in terms of valued hyperfields for finitely ramified valued fields. We first recall basic facts on coarsenings of valuations.

\begin{rem/def}\label{rem/def:coarse_valuation}\cite{PR}
Let $(K,\nu, k, \Gamma)$ be valued field. Let $\Gamma^{\circ}$ be a convex subgroup of $\Gamma$ and $\dot{\nu}:K\setminus\{0\}\longrightarrow \Gamma/\Gamma^{\circ}$ be a map sending $x(\neq 0)\in K$ to $\nu(x)+\Gamma^{\circ}\in \Gamma/\Gamma^{\circ}$. The map $\dot{\nu}$ is a valuation, called {\em a coarse valuation} of $\nu$ with respect to $\Gamma^{\circ}$. The residue field $\cfd$, called {\em the core field} of $(K,\nu)$ with respect to $\Gamma^{\circ}$, of $(K,\dot{\nu})$ forms a valued field equipped with a valuation $\cnu$ induced from $\nu$ and the value groups $\Gamma^{\circ}$. More precisely, the valuation $\cnu$ is defined as follows: Let $\pr_{\dnu} : R_{\dnu}\longrightarrow \cfd$ be the canonical projection map and let $x\in R_{\dnu}$. If $x^{\circ}:=\pr_{\dnu}(x)\in \cfd\setminus\{0\}$, then $\cnu(x^{\circ}):=\nu(x)$. And $x^{\circ}=0\in \cfd$ if and only if $\nu(x)>\Gamma^{\circ}$, that is, $\nu(x)>\gamma$ for all $\gamma\in \Gamma^{\circ}$. If $K$ is of characteristic $0$ and $\Gamma^{\circ}$ is non-trivial, then $(K,\dnu)$ is always of equal characteristic $(0,0)$.
\end{rem/def}

\begin{fact}\label{fact:coarse_valuation}
Let $(K,\nu,\Gamma)$ be valued field. Let $\dnu$ be the coarse valuation and $K^{\circ}$ be the core field with respect to $\Gamma^{\circ}$ for a non-trivial convex subgroup $\Gamma^{\circ}$ of $\Gamma$.
\begin{enumerate}
	\item Let $R_{\nu}$, $R_{\dnu}$, and $R_{\cnu}$ be the valuation rings of $(K,\nu)$, $(K,\dnu)$, and $(\cfd,\cnu)$ respectively. Then $(\pr_{\dnu})^{-1}(R_{\cnu})=R_{\nu}$.
	\item If $(K,\nu)$ is finitely ramified, then $(\cfd,\cnu)$ is finitely ramified, and $K$ and $\cfd$ have the same ramification index.
	\item If $(K,\nu)$ is finitely ramified and $\aleph_1$-saturated, and $\Gamma^{\circ}$ is the smallest non-trivial convex subgroup, then $(\cfd,\cnu)$ is complete.
\end{enumerate}
\end{fact}

\begin{definition}\label{def:restricted_hf}
Let $(K,\nu,\Gamma)$ be valued field. Let $\Gamma^{\circ}\subset \Gamma$ be a non-trivial convex subgroup and $\gamma\in (\Gamma^{\circ})_{>0}$. Let $0_{\Gamma^{\circ}}^{\gamma}:=\{\alpha\in \H_{\gamma}(K)|\ \hn(\alpha)>\Gamma^{\circ} \}$. If $\gamma$ is obvious, we write $0_{\Gamma^{\circ}}^{\gamma}=0_{\Gamma^{\circ}}$. Define $\H_{\gamma}^{\times}(K,\Gamma^{\circ}):=\{\alpha\in\H_{\gamma}(K): \hn(\alpha)\in \Gamma^{\circ}\}$ and $\H_{\gamma}(K,\Gamma^{\circ}):=\H_{\gamma}^{\times}(K,\Gamma^{\circ})\cup \{ 0_{\Gamma^{\circ}}^{{\gamma}} \}$.
\end{definition}

\begin{remark}
Let $\Gamma^{\circ}$ be a non-trivial convex subgroup of $\Gamma$. Then $\H_{\gamma}(K,\Gamma^{\circ})$ forms a valued hyperfield.
\end{remark}
\begin{proof}
Let $(\H_{\gamma}(K),\hp,\hm,\hn)$ be the valued $\gamma$-hyperfield of $K$. First, $\H_{\gamma}^{\times}(K,\Gamma^{\circ})$ is a multiplicative subgroup of $(\H_{\gamma}^{\times}(K),\hm)$ because $\Gamma^{\circ}$ is a subgroup of $\Gamma$. Note that for $\alpha,\beta\in \H_{\gamma}^{\times}(K)$, if there is $x\in \alpha\hp\beta$ such that $\nu(x)>\Gamma^{\circ}$, then $0_{\Gamma^{\circ}}\subset \alpha\hp\beta$($\dagger$). For $\alpha=[a]$ and $\beta=[b]$, $\bigcup \alpha\hp\beta=(a+b)+\fm^{\gamma+\min\{\nu(a),\nu(b)\}}$ by Lemma \ref{lem:basic_on_hyperfield}(2). Suppose $x\in \bigcup \alpha\hp\beta$ such that $\nu(x)>\Gamma^{\circ}$. If $\nu(a+b)<\gamma+\min\{\nu(a),\nu(b)\}$, then $\nu(x)=\nu(a+b)<\gamma+\min\{\nu(a),\nu(b)\}\in\Gamma^{\circ}$ because $\gamma,\nu(a),\nu(b)\in \Gamma^{\circ}$. So we have that $\nu(a+b)\ge \gamma+\min\{\nu(a),\nu(b)\}$ and $\bigcup\alpha\hp\beta=\fm^{\gamma+\min\{\nu(a),\nu(b)}\supset 0_{\Gamma^{\circ}}$. We define a multivalued operation $\hp^{\Gamma^{\circ}}$ on $\H_{\gamma}(K,\Gamma^{\circ})$ as follows: For $\alpha,\beta\in \H_{\gamma}^{\times}(K)$,
\begin{itemize}
	\item $\alpha\hp^{\Gamma^{\circ}}0_{\Gamma^{\circ}}=\alpha=0_{\Gamma^{\circ}}\hp^{\Gamma^{\circ}}\alpha$.
	\item $\alpha\hp^{\Gamma^{\circ}}\beta=\begin{cases}
\alpha\hp\beta& \mbox{if } \forall x\in \bigcup \alpha\hp\beta,\ \nu(x)\in \Gamma^{\circ}\\
(\alpha\hp\beta)\cap(\H_n)^{\times}(K)\cup\{0_{\Gamma^{\circ}}\}& \mbox{if } \exists x\in \bigcup \alpha\hp\beta,\ \nu(x)>\Gamma^{\circ}
	\end{cases}$.
\end{itemize}
By ($\dagger$) and covexity of $\Gamma^{\circ}$, $\hp^{\Gamma^{\circ}}$ is well-defined. Define a map $\hn^{\Gamma^{\circ}}$ on $\H_{\gamma}(K,\Gamma^{\circ})$ as follows: For $\alpha\in \H_{\gamma}(K,\Gamma^{\circ})$, $\hn^{\Gamma^{\circ}}(\alpha)=\hn(\alpha)$ if $\alpha\neq 0_{\Gamma^{\circ}}$ and $\hn^{\Gamma^{\circ}}(\alpha)=\infty$ if $\alpha=0_{\Gamma^{\circ}}$. Then $(\H_{\gamma}(K,\Gamma^{\circ}),\hp^{\Gamma^{\circ}}, \hm, \hn^{\Gamma^{\circ}})$ forms a valued hyperfield.
\end{proof}

\begin{lemma}\label{lem:coarse_valuation}
Let $(K,\nu,\Gamma)$ be valued field. Let $\dnu$ be the coarse valuation and $K^{\circ}$ be the core field with respect to $\Gamma^{\circ}$ for a non-trivial convex subgroup $\Gamma^{\circ}$ of $\Gamma$. For each $\gamma\in \Gamma^{\circ}_{>0}$, $\H_{\gamma}(K,\Gamma^{\circ})$ and $\H_{\gamma}(K^{\circ})$ are isomorphic.
\end{lemma}
\begin{proof}
Consider a map $f: \H_{\gamma}(K,\Gamma^{\circ})\rightarrow \H_{\gamma}(K^{\circ})$ defined as follows: For $\alpha\in \H_{\gamma}(K,\Gamma^{\circ})$, $f(\alpha)=[a^{\circ}]$ if $\alpha=[a]$, and $f(\alpha)=0$ if $\alpha=0_{\Gamma^{\circ}}$. The map $f$ is well-defined. Suppose $[a]=[b]\in\H_{\gamma}(K,\Gamma^{\circ})$. Then $a/b\in (1+\fm^{\gamma})$ and $\nu(a)=\nu(b)\in \Gamma^{\circ}$. So, we have that $(a/b)^{\circ}=a^{\circ}/b^{\circ}\in (1+(\fm^{\circ})^{\gamma})$, and it implies $[a^{\circ}]=[b^{\circ}]$. And it is clear that for all $x\in K$ with $\nu(x)>\Gamma^{\circ}$, $x^{\circ}=0$. It is routinely to check that $f$ is a homomorphism.
\end{proof}

\noindent We recall the following facts before proving a relative completeness theorem.
\begin{fact}\label{fact:KS_iso_theorem}{(Keisler-Shelah Isomorphism Theorem)}
Let $\CM$ and $\CN$ be two first order structures. If $\CM\equiv \CN$, then there is a ultrafilter $\CU$ on an infinite set $I$ such that $$\CM^{\CU}\cong \CN^{\CU},$$ where $\CM^{\CU}$ and $\CN^{\CU}$ are the ultrapowers of $\CM$ and $\CN$ with respect to $\CU$.
\end{fact}

\begin{fact}\cite[Lemma 1.5]{Be}\label{fact:def_resiue_ring}
Let $(K,\nu)$ be a finitely ramified henselian valued field of mixed characteristic $(0,p)$. Then the valuation ring $R(K)$ of $(K,\nu)$ is definable by the formula $$\phi_q(x)\equiv \exists y\ y^{q}=1+px^{q}$$ for some $q>0$ such that $p\not|q$ and $q>e_{\nu}(p)$.
\end{fact}

\begin{theorem}\label{thm:AKE_hyperfield}
Let $(K_1,\nu_1,k_1,\Gamma_1)$ and $(K_2,\nu_2,k_2,\Gamma_1)$ be finitely ramified henselian valued fields of mixed characteristic $(0,p)$. Suppose $k_1$ and $k_2$ are perfect fields. Let $n> e_{\nu_2}(p)(1+e_{\nu_1}^2(p))$. The following are equivalent:
		\begin{enumerate}
			\item $K_1\equiv K_2$.
			\item $R_{n}(K_1)\equiv R_{n}(K_2)$ and $\Gamma_1\equiv \Gamma_2$.
			\item $\H_{n}(K_1)\equiv_{\H_{n}(\{p\})} \H_{n}(K_2)$.
		\end{enumerate}
\end{theorem}
\begin{proof}
For $(1)\Leftrightarrow (2)$, see Fact \ref{fact:fined_AKE_Rn}. It is clear that $(1)$ implies $(3)$. We show $(3)$ implies $(1)$. Suppose $(3)$ holds. By Fact \ref{fact:KS_iso_theorem}, we may assume that $\H_{n}(K_1)\cong_{\H_{n}(\BZ)} \H_{n}(K_2)$. By Remark \ref{rem:vhf_valuegroup}, $\Gamma_1\cong \Gamma_2$ and we may assume that $\Gamma_1=\Gamma_2=\Gamma$. By taking ultrapowers of $K_1$ and $K_2$ with respect to a nonprincipal ultrafilter on $\BN$, we may assume that $K_1$ and $K_2$ are $\aleph_1$-saturated. Let $\dnu_1$ and $\dnu_2$ be the coarse valuation of $\nu_1$ and $\nu_2$ with respect to the smallest non-trivial convex subgroup $\Gamma^{\circ}$. We have two valued fields $(K_1,\dnu_1)$ and $(K_2,\dnu_2)$ of equal characteristic $(0,0)$ with residue fields $K_1^{\circ}$ and $K_2^{\circ}$ respectively. By Fact \ref{fact:coarse_valuation}(2) and (3). $(K_1^{\circ},\cnu_1)$ and $(K_2^{\circ},\cnu_2)$ are complete discrete valued fields of mixed characteristic $(0,p)$. Note that the value groups of $(K_1,\dnu_1)$ and $(K_2,\dnu_2)$ are $\Gamma/\Gamma^{\circ}$. Since $\H_n(K_1)\cong_{\H_n(\BZ)}\H_n(K_2)$, we have that $\H_n(K_1^{\circ})\cong_{\H_n(\BZ)} \H_n(K_1^{\circ})$ by Lemma \ref{lem:coarse_valuation}. Since $e_{\cnu_i}(p)=e_{\nu_i}(p)$ for $i=1,2$, we have the $\cfd_1\cong \cfd_2$ by Corollary \ref{cor:criteria_isom}. By Fact \ref{fact:AKE}, we have that $(K_1,\dnu_1)\equiv (K_2,\dnu_2)$. To show that $(K_1,\nu_1)\equiv(K_2,\nu_2)$, it is enough to show that the valuation rings $R_{\nu}(K_1)$ of $(K_1,\nu_1)$ and $R_{\nu}(K_2)$ of $(K_2,\nu_2)$ are definable in $(K_1,\dnu_1)$ and $(K_2,\dnu_2)$ by the same formula. Recall the following result on a definability of a residue ring. Take $l>0$ large enough so that $q:=p^l+1>\max\{e_{\nu_1}(p),e_{\nu_2}(p)\}$. By Fact \ref{fact:def_resiue_ring}, $\phi_q(x)$ defines the residue rings $R_{\cnu}(K_1)$ and $R_{\cnu}(K_2)$ of $(\cfd_1,\cnu_1)$ and $(\cfd_2,\cnu_2)$. By Fact \ref{fact:coarse_valuation}(2), the valuation rings $R_{\nu}(K_1)$ and $R_{\nu}(K_2)$ are definable by the same formula in $(K_1,\dnu_1)$ and $(K_2,\dnu_1)$ so that $(K_1,\nu_1)\equiv (K_2,\nu_2)$.
\end{proof}
\noindent By Theorem \ref{thm:tame_vf_vhf} and the proof of $(3)\Rightarrow (1)$ of Theorem \ref{thm:AKE_hyperfield}, we have the following result.
\begin{corollary}\label{cor:AKE_hyperfield_tame}
Let $(K_1,\nu_1,k_1,\Gamma_1)$ and $(K_2,\nu_2,k_2,\Gamma_1)$ be finitely tamely ramified henselian valued fields of mixed characteristic $(0,p)$. Suppose $k_1$ and $k_2$ are perfect fields. The following are equivalent:
\begin{enumerate}
	\item $K_1\equiv K_2$.
	\item $\H_1(K_1)\equiv_{\H_1(\{p\})} \H_1(K_2)$.
\end{enumerate}
\end{corollary}
\noindent For local fields of mixed characteristic, they are elementary equivalent if and only if they are isomorphic. So we have the following corollary.
\begin{corollary}
Let $K_1$ and $K_2$ be local fields of mixed characteristic. Let $n>e_{\nu_2}(p)(1+e_{\nu_1}^2(p))$. The followings are equivalent:
\begin{enumerate}
	\item $\H_{n}(K_1)\equiv_{\H_{n}(\{p\})} \H_{n}(K_2)$.
	\item $\H_{n}(K_1)\cong_{\H_{n}(\{p\})}  \H_{n}(K_2)$.
\end{enumerate}
\end{corollary}

\end{document}